\documentclass[11pt, letterpaper]{amsart}
\usepackage{mathrsfs}
\usepackage[margin=3cm]{geometry}
\usepackage{amssymb}
\usepackage{cite}
\usepackage[all]{xy}
\usepackage{graphicx}
\usepackage[dvipsnames]{xcolor}
\usepackage{textcomp}
\usepackage{charter}
\usepackage[vflt]{floatflt}
\usepackage{mathtools}
\usepackage{enumerate}
\usepackage{wasysym}

\usepackage{pdflscape}
\usepackage{comment}
\usepackage{bm}
\usepackage[pdftex,colorlinks]{hyperref}
\hypersetup{
bookmarksnumbered,
pdfstartview={FitH}}
\usepackage{tikz-cd}
\usepackage{nicematrix}

\usepackage{blkarray}
\usepackage{amsmath}
\usepackage{amsthm}
\usepackage{esint}

\usepackage{setspace}

\usepackage{amsfonts}
\usepackage{eucal}
\usepackage{latexsym}
\usepackage{mathdots}
\usepackage{mathrsfs}
\usepackage{enumitem}

\newtheorem{thm}{Theorem}[section]

\newtheorem{lem}[thm]{Lemma}
\newtheorem{lemma}[thm]{Lemma}

\newtheorem{cor}[thm]{Corollary}
\newtheorem*{thm*}{Theorem}
\newtheorem*{cor*}{Corollary}
\newtheorem*{prop*}{Proposition}

\theoremstyle{definition}
\newtheorem{defn}[thm]{Definition}

\theoremstyle{remark}
\newtheorem{remark}[thm]{Remark}

\newtheorem{quest}[thm]{Question}
\numberwithin{equation}{section}

\newcommand{\be}{\begin{equation}}
\newcommand{\ee}{\end{equation}}
\def\ba{\begin{eqnarray*}}
\def\ea{\end{eqnarray*}}

\newcommand{\bi}{\begin{itemize}}
\newcommand{\ei}{\end{itemize}}
\newcommand{\bn}{\begin{enumerate}}
\newcommand{\en}{\end{enumerate}}
\newcommand{\bbm}{\begin{bmatrix}}
\newcommand{\ebm}{\end{bmatrix}}
\newcommand{\bpm}{\begin{pmatrix}}
\newcommand{\epm}{\end{pmatrix}}
\newcommand{\bsm}{\left ( \begin{smallmatrix}}
\newcommand{\esm}{\end{smallmatrix} \right) }

\newcommand{\mr}{\ensuremath{\mathrm}}
\newcommand{\scr}{\ensuremath{\mathscr}}

\newcommand{\ov}{\ensuremath{\overline}}
\newcommand{\sm}{\ensuremath{\setminus}}
\newcommand{\wt}{\ensuremath{\widetilde}}

\newcommand{\Om}{\ensuremath{\Omega}}

\newcommand{\La}{\ensuremath{\Lambda }}

\newcommand{\om}{\ensuremath{\omega}}

\newcommand{\eps}{\ensuremath{\epsilon }}

\def\C{\mathbb{C}}
\def\R{\mathbb{R}}

\def\N{\mathbb{N}}
\def\B{\mathbb{B}}

\def\fr{\mathfrak{r}}

\newcommand{\cO}{\mathcal{O}}
\newcommand{\cA}{\mathcal{A}}
\newcommand{\bW}{\mathbb{W}}
\newcommand{\cE}{\mathcal{E}}

\newcommand{\fps}{\C\langle \langle z \rangle \rangle}
\def\word{\mathbb{W} _d}

\def\fz{\mathfrak{z}}
\def\bH{\mathbb{H}}

\def\A{\mathbb{A} _d}

\def\fp{\mathbb{C} \langle \fz \rangle }
\def\fpd{\mathbb{C} \langle \fz _1, \cdots, \fz _d \rangle }
\def\fps{\mathbb{C} \langle \! \langle  \fz  \rangle \! \rangle}
\def\fpsd{\mathbb{C} \langle \! \langle  \fz _1, \cdots, \fz _d  \rangle \! \rangle}
\def\mrt{\mathrm{t}}

\def\hardy{\mathbb{H} ^2 _d}
\def\mult{\mathbb{H} ^\infty _d}
\newcommand{\ip}[2]{\ensuremath{\langle {#1} , {#2} \rangle}}
\def\nbdom{\mr{Dom} \, }

\def\nbran{\mr{Ran} \, }
\def\nbker{\mr{Ker} \, }
\def\nbdim{\mr{dim} \, }

\def\fskew{\C \ \mathclap{\, <}{\left( \right.}   \fz  \mathclap{  \, \, \, \, \, >}{\left. \right)} \, \, }

\def\skewfps{\C \ \mathclap{\, <}{\left( \right.} \!  \mathclap{\, <}{\left( \right.}  \fz  \mathclap{  \, \, \, \, \, >}{\left. \right)} \! \mathclap{  \, \, \, \, \, >}{\left. \right)} \, \,}

\def\fskewm{\C   \ \mathclap{\, <}{\left( \right.}   \fz _1, \cdots, \fz _m  \mathclap{  \, \, \, \, \, >}{\left. \right)} \, \, }

\def\fskewk{\C   \ \mathclap{\, <}{\left( \right.}   \fz _1, \cdots, \fz _k  \mathclap{  \, \, \, \, \, >}{\left. \right)} \, \, }

\def\fskewd{\C   \ \mathclap{\, <}{\left( \right.}   \fz _1, \cdots, \fz _d  \mathclap{  \, \, \, \, \, >}{\left. \right)} \, \, }

\def\ratfps{\C _0 \ \mathclap{\, <}{\left( \right.}  \fz  \mathclap{ \, \, \, \, \, >}{\left. \right)} \, \, }

\def\ratfpsm{\C _0  \ \mathclap{\, <}{\left( \right.}  \fz _1, \cdots, \fz _m  \mathclap{ \, \, \, \, \, >}{\left. \right)} \, \, }

\def\bdn{\mathbb{B} ^{(n\times n)\cdot d}}
\def\bdm{\mathbb{B} ^{(m\times m)\cdot d}}
\def\cdm{\mathbb{C} ^{(m\times m)\cdot d}}
\def\cdn{\mathbb{C} ^{(n\times n)\cdot d}}

\def\cdmr{\mathbb{C} ^{(m\times m)\times d}}

\def\cH{\mathcal{H}}
\def\cJ{\mathcal{J}}
\def\cK{\mathcal{K}}
\def\cW{\mathcal{W}}

\def\ncu{\mathbb{C} ^{(\N \times \N) \cdot d}}
\def\ncuext{\mathbb{C} ^{(\aleph _0 \times \aleph _0 ) \cdot d}}

\def\rball{\mathbb{B} ^{(\N \times \N) \cdot d}}

\newcommand{\abs}[1]{\lvert #1 \rvert}
\newcommand{\Abs}[1]{\left\lvert #1 \right\rvert}
\newcommand{\norm}[1]{\lVert #1 \rVert}
\newcommand{\normBig}[1]{\Big\lVert #1 \Big\rVert}
\newcommand{\Norm}[1]{\left\lVert #1 \right\rVert}
\DeclareMathOperator{\GL}{GL}

\DeclareMathOperator{\sgrm}{SGRM}
\DeclareMathOperator{\tr}{tr}
\DeclareMathOperator{\GM}{GM}
\DeclareMathOperator{\UD}{UD}
\DeclareMathOperator{\rank}{rank}
\newcommand{\gentup}[1]{{\underline{\Xi}^{#1}}}
\newcommand{\genvars}[1]{{\underline{\xi}^{#1}}}
\newcommand{\canr}[1]{{\C[\mkern-2.7mu [\genvars{#1};r]\mkern-2.7mu ] }}
\newcommand{\cmeror}[1]{{\C_r(\mkern-2.7mu (\genvars{#1})\mkern-2.7mu ) }}

\newcommand{\meric}[1]{{\color{Turquoise} #1}} 

\title{Rings of non-commutative functions and their fields of fractions}

\author{M\'eric L. Augat}
\address{James Madison University}
\thanks{M. L. A was partially supported by NSF grant DMS-2155033}
\email{augatmx@jmu.edu}

\author{Robert T.W. Martin}
\address{\footnotesize University of Manitoba}
\thanks{R. T. W. M. was supported by NSERC grant 2020-05683}
\email{Robert.Martin@umanitoba.ca}

\author{Eli Shamovich}
\address{\footnotesize Ben-Gurion University of the Negev}
\thanks{E. S. was partially supported by BSF grant 2022235}
\email{shamovic@bgu.ac.il}


\date{\today}

\begin{document}

\begin{abstract}
Semi-free ideal rings, or \emph{semifirs}, were introduced by Paul M. Cohn to study universal localizations in the non-commutative setting. We provide new examples of semifirs consisting of analytic functions in several non-commuting variables. These examples arise canonically in free analysis by completing the free algebra in a topology of ``uniform convergence on balls of finite radii'' in the non-commutative universe of tuples of square matrices of any finite size. Namely, we prove that the rings of all formal power series in $d \in \N$ non-commuting variables with ``radius of convergence" at least $r \in (0, +\infty]$ are semifirs. In particular, the ring of entire non-commutative (NC) functions, $\scr{O}_d$, is a semifir.  Any semifir admits a universal skew field of fractions; applying this to $\scr{O}_d$ yields the universal skew field of (global) NC meromorphic functions, $\scr{M} _d$. We extend a celebrated theorem of Amitsur negating the existence of non-trivial ``NC rational identities" to prove that there are no non-trivial ``NC meromorphic identities" in $\scr{M} _d$. This allows us to identify any $f \in \scr{M} _d$ with a bona fide analytic non-commutative function in the NC universe whose domain is analytic-Zariski dense at all sufficiently high matrix levels. 

We further develop several applications: Every finitely--generated right (or left) ideal in $\scr{O}_d$ is closed; this yields an analytic extension of G. M. Bergman's nullstellensatz for the free algebra. Any $f \in \scr{M} _d$ has a well-defined domain and evaluations in a large class of stably-finite topological algebras, including finite $C^*$-algebras, generalizing a result of Cohn for NC rational functions. This enables us to extend a random matrix approximation result of Haagerup and Thorbj{\o}rnsen for free polynomials to the setting of NC meromorphic functions.
\end{abstract}
\maketitle

\section{Introduction}

This paper studies algebras of non-commutative functions and their ring-theoretic properties. The theory of non-commutative functions lies in the intersection of free analysis and pure algebra. This is exemplified in this paper in which we employ (largely) analytic techniques to prove algebraic results. Namely, we apply operator algebraic and functional analytic technologies, in particular the theory of non-commutative Hardy spaces of NC analytic functions, to prove that certain rings of non-commutative analytic functions presented as power series in several NC variables are semi-free ideal rings, a class of rings that admit universal localizations. We apply this to extend a result of Haagerup and Thorbj{\o}rnsen in random matrix and free probability theory.

Non-commutative function theory originated in the work of J.L. Taylor in multivariate spectral theory and functional calculus for tuples of non-commuting operators \cite{Taylor, Taylor2}. A non-commutative function in several non-commuting variables is a function on tuples of square matrices of any fixed, finite size, taking values in square matrices, and which has natural invariance properties enjoyed by free polynomials, $p \in \fp = \C \langle \fz_1,\cdots,\fz_d \rangle$. Here, $\fp$ denotes the \emph{free algebra} over $\C$ on $d$ generators. We will call elements of $\fp$ \emph{non-commutative} or \emph{free polynomials}. A function, $f$, on a direct sum-closed subset of the complex \emph{NC universe} of $d-$tuples of square matrices over $\C$ of any fixed finite size is an \emph{NC function} if it: (i) respects the grading (matrix size), (ii) respects direct sums and (iii) respects joint similarities, see Subsection \ref{ss:NCfun} for definitions and details. It is readily checked that any $p \in \fp$ is an NC function on the entire NC universe. These three axioms defining NC functions are surprisingly rigid: Any NC function that is `continuous' or even `locally bounded' in an `open' NC set is automatically holomorphic in the sense that it is Fr\'echet differentiable at any point in its domain and analytic in the sense that it has a convergent Taylor-type power series expansion about any point of its domain with non zero ``radius of convergence". (In order to precisely define the punctuated terms in the previous sentence, we need to introduce a topology on the NC universe, which we will do in Subsection \ref{ss:NCfun}. The topology that is most relevant for our investigations here is called the \emph{uniform topology} \cite[Section 7.2]{KVV}.)

Beyond free polynomials and NC rational functions -- which are NC functions constructed by applying the arithmetic operations of summation, addition and inversion to the free algebra -- many examples of NC functions are given by formal power series in several non-commuting variables. Any such free formal power series (FPS) in several non-commuting indeterminates defines a uniformly analytic NC function on a `ball' centred at the origin, $0 = (0, \cdots, 0) \in \C ^{1\times d}$, of the NC universe if it has a positive ``radius of convergence". Namely, let $\fps = \fpsd$ denote the ring of free FPS in the NC formal variables $\fz := (\fz_1, \cdots, \fz _d)$ with complex coefficients. Any $f \in \fps$ can be written
$$ f(\fz) = \sum _{\om \in \word} \hat{f} _\om \fz ^\om; \quad \quad \hat{f} _\om \in \C, $$ 
where $\om= i_1 \cdots i_n$ is any \emph{word} comprised of \emph{letters} $i_j$, chosen from the \emph{alphabet}, $i_j \in \{1, \cdots, d\}$. Here, $\word$ denotes the \emph{free monoid}, the set of all such words in $d$ letters. We also consider the \emph{empty word}, $\emptyset$, consisting of no letters, and the free monomials, $\fz ^\om$, are then defined in the obvious way: $\fz ^\om = \fz _{i_1} \cdots \fz _{i_n}$ if $\om = i_1 \cdots i_n$ and $\fz ^\emptyset =:1$.  G. F. Popescu \cite{Pop-freeholo} introduced an analogue of the Cauchy--Hadamard radius of convergence formula for such free FPS, $f$, given by
$$ \frac{1}{R_f} := \limsup _{\ell \rightarrow \infty} \sqrt[\uproot{3} 2\ell]{\sum _{|\om | = \ell} | \hat{f} _\om | ^2}, $$ see also \cite{Luminet-PI} and \cite{KVV}. In the above, $|\om | \in \N \cup \{ 0 \}$ denotes the \emph{length} of $\om \in \word$, \emph{i.e.,} the number of letters it contains. Given any $r \in (0, +\infty]$, we study the rings, $\scr{O} _d (r)$, of all free formal power series with radius of convergence at least $r$. We write $\scr{O} _d := \scr{O} _d (+\infty) = \cap _{r>0} \scr{O} _d (r)$, for the ring of NC entire functions. As proven in \cite[Theorem 1.1]{Pop-freeholo}, if $f \in \fps$, has radius of convergence $R_f >0$, then $f(X)$ converges absolutely and uniformly in an NC row-ball, $R_f \cdot \rball$, centred at the origin, $0=(0, \cdots, 0) \in \C ^{1\times d}$, of the $d-$dimensional NC universe, consisting of all square matrix $d-$tuples, $X = (X_1, \cdots, X_d) \in \C ^{n\times n} \otimes \C ^{1\times d}$, which define bounded linear maps from $\C ^n \otimes \C ^d$ into $\C^n$, with norm less than $R_f$. Hence, any $f \in \scr{O} _d (r)$ defines a (uniformly) analytic NC function in $r \cdot \rball$. In particular, $\scr{O}_d$ consists of holomorphic and analytic non-commutative functions on the entire NC universe and is, in a very natural sense, a canonical multivariate extension of the ring of entire functions in $\C$ to several non-commuting variables.

In pure non-commutative algebra, semi-free ideal rings were introduced by P.M. Cohn in his study of localizations and universal skew fields of fractions of non-commutative rings, see \cite{Cohn} for a full development of this theory. A ring, $\scr{R}$, is a \emph{semi-free ideal ring} (semifir for short), if every finitely generated right ideal of $\scr{R}$ is free as a right $\scr{R}-$module. This notion is left-right symmetric. Cohn showed that every semi-free ideal ring admits a (necessarily unique) universal skew field of fractions. There are not many known examples of semi-free ideal rings, and many non-commutative rings do not have universal localizations. (Certainly, any ring with zero divisors cannot be embedded into a skew field.) Among the few known examples are the free algebra, the ring $\fps = \C \langle \! \langle \fz _1, \cdots, \fz _d \rangle \! \rangle$ of formal power series in several non-commuting variables, the sub-ring of algebraic free FPS, the group ring of the free group and group rings over locally free groups (see, for example,  \cite{Cohn} and \cite{Dicks-grp}). Some more recent examples of universal localizations and of semifirs were constructed by Klep, Vinnikov, and Vol\v{c}i\v{c} in \cite{KVV-local, KVVol-multipartite} and by the current authors in \cite{AugMarS}. The first main result of this paper provides a new family of examples of semifirs of uniformly analytic NC functions:

\begin{thm*}[Theorem \ref{thm:entires_semifir}]
    The rings $\scr{O}_d (r)$ are semi-free ideal rings for any $r \in (0, +\infty]$.
\end{thm*}
Our proof of the above theorem relies heavily on the operator--algebraic ``Non-commutative Hardy space theory" initiated by Popescu and Davidson--Pitts, in which the full Fock space over $\C^d$ is interpreted as the ``free Hardy space" of square--summable formal power series in $d$ NC variables \cite{Pop-freeholo,Pop-entropy,Pop-ncdisk,DP-inv}. In particular, we will employ the non-commutative inner--outer factorization \cite{DP-inv,Pop-char,Pop-multi} for the free Hardy space, and extensions of methods in \cite{JMS-ncBSO} in conjunction with the structure of $\scr{O}_d (r)$ as an inverse limit of non-commutative disk algebras on balls of increasing radii \cite{Arens-dense_lim}. 

It follows that each $\scr{O} _d (r)$, $r\in (0, +\infty]$ has a universal skew field of fractions, which we denote by $\scr{M} _d (r)$. In particular, we will call $\scr{M} _d := \scr{M} _d (\infty)$ the field of (global) NC (uniformly) meromorphic functions -- c.f. the \textit{local} field of (uniformly) meromorphic NC germs at $0\in \C ^{1\times d}$ investigated in \cite{KVV-local}. Elements of $\scr{M} _d (r)$ are certain equivalence classes of non-commutative rational expressions in elemends of $\scr{O} _d (r)$, in order to prove that elements of $\scr{M} _d (r)$ are bona fide analytic non-commutative functions, it is necessary to rule out the possibility of non-trivial NC meromorphic identities in the sense of Amitsur \cite{Amitsur}. Namely, there could exist, a priori, a non-zero element $f \in \scr{M} _d$ that vanishes identically everywhere that it is defined in the NC universe. Hence, it would not be possible to identify such $0 \neq f$ or $f ^{-1}$ with holomorphic non-commutative functions in the sense of NC function theory. 

In Section \ref{sec:usfield}, we prove that each of the natural embeddings $\scr{O} _d \hookrightarrow \scr{O} _d (R) \hookrightarrow \scr{O} _d (r) \hookrightarrow \scr{O} _d (0 ) \hookrightarrow \fps$, $r<R \in (0, +\infty)$, is \emph{honest}, \emph{i.e.} it preserves the ``inner" rank of any square matrix over these semifirs. Here, $\scr{O} _d (0)$, denotes the semifir of (uniformly) analytic NC germs at $0=(0, \cdots, 0)$, the origin of the NC universe, as constructed and studied in \cite{KVV-local}. By results of Cohn, these embeddings then extend to embeddings of the universal skew fields of fractions of these semifirs. This enables us to prove that there are no non-trivial meromorphic identities.

\begin{thm*}[NC meromorphic Amitsur theorem, Corollary \ref{cor:Amitsur Thm for Mdr}]
There are no non-trivial non-commutative meromorphic identities. Every $f \in \scr{M} _d (r)$ has a domain in the uniformly open NC row-ball, $r \cdot \rball$, of radius $r>0$ in the NC universe that is analytic-Zariski dense at sufficiently high levels, and $f$ vanishes identically on this domain if and only if $f=0 \in \scr{M} _d (r)$. Every $f \in \scr{M} _d (r)$ is a uniformly analytic NC function on its (uniformly open) domain.  
\end{thm*}

We further obtain an analytic extension of a non-commutative nullstellensatz for the free algebra due to G. M. Bergman to each of our semifirs, $\scr{O} _d (r)$, see \cite{HelMcC-positivss}. To this end, define the directional variety of $S \subseteq \scr{O}_d(r)$ as
\[
\scr{Z}(S) = \bigsqcup_{n \leq \aleph_0} \left\{ \left. (X,y) \in r \cdot \bdn \times (\C ^n \sm \{ 0 \}) \right| \  y \perp \nbran f(X) \ \forall f \in S \right\}.
\]
In the above, $\bdn$, denotes the $n$th level of the NC unit row-ball consisting of all $X = (X_1, \cdots, X_d) \in \C ^{n\times n} \otimes \C ^{1\times d}$ so that $\| X \| _{\scr{B} (\C ^n \otimes \C ^d, \C^n )} <1$. Note that we also include the `infinite level', 
$\B ^{(\aleph _0 \times \aleph _0) \cdot d}$, consisting of all $T=(T_1, \cdots, T_d) \in \scr{B} (\cH) \otimes \C ^{1\times d}$ for which $\| T \| _{\scr{B} (\cH \otimes \C ^d, \cH )} <1$, where $\cH$ is a complex and separable Hilbert space. Such $X,T$ are called strict \emph{row contractions}. As in the classical case, for every $$\scr{Z} \subseteq \bigsqcup_{n\leq \aleph_0} r \cdot \bdn \times (\C^n \setminus \{0\}),$$ we can define a right ideal of analytic NC functions
\[ I (\scr{Z}) = \left\{f \in \scr{O}_d (r) \left| \ y \perp \nbran f(X)  \ \forall \ (X,y) \in \scr{Z} \right. \right\}.
\]

\begin{thm*}[Analytic Bergman Nullstellensatz, Theorem \ref{thm:bergman}]
Let $J \vartriangleleft \scr{O}_d (r)$, $r \in (0, +\infty]$ be a finitely--generated right ideal. Then, $J$ is closed and $J = I(\scr{Z}(J))$.
\end{thm*}

In Section \ref{sec:application}, we extend a theorem of P.M. Cohn for evaluation of non-commutative rational functions in any stably-finite algebra. Here, the \emph{free skew field} of all NC rational functions, $\fskew := \fskewd$, is the universal skew field of fractions of the free algebra, $\fp = \fpd$ \cite{Amitsur,Cohn}. Namely, we prove that NC meromorphic elements in $\scr{M} _d$ have well-defined evaluations and domains in any stably finite algebra, $\scr{A}$, that admits well-defined unital ring homomorphisms $\varphi: \scr{O}_d \rightarrow \scr{A}$; in particular in any stably finite $C^*-$algebra. This result enables us to develop concrete applications of our results to D.-V. Voiculescu's free probability theory and random matrix theory. Namely, we extend the results of Haagerup and Thorbj{\o}rnsen for the free algebra \cite{Hup-realize2} and of Yin \cite{Yin-strong_conv} for the free skew field, to the larger field of NC meromorphic functions. To describe these results in more detail, let $s_1,\cdots, s_d$ denote the real parts of the left creation operators on the full Fock space over $\C ^d$. This $d-$tuple is the prototypical example of free semicircular random variables in free probability theory.

\begin{thm*}
Let $(\Omega,P)$ be a probability space. Let $X_1^{(n)},\cdots,X_d^{(n)}$ be a sequence of $d$-tuples of independent random matrices of size $n \times n$ with independent and identically distributed Gaussian entries over $\Omega$ with mean zero and variance $\frac{1}{n}$. Then, given any $f \in \scr{M} _d$, if $f(s_1,\cdots,s_d)$ is defined, then
$f(X_1^{(n)}(\omega),\cdots,X_d^{(n)}(\omega))$ is well defined for sufficiently large $n \in \N$ and almost every $\omega \in \Omega$, and
\[
\lim_{n \to \infty} \|f(X_1^{(n)}(\omega),\cdots,X_d^{(n)}(\omega))\| = \|f(s_1,\cdots,s_d) \|.
\]
\end{thm*}

\subsection{Structure of the paper}

Section \ref{sec:prelims} provides the necessary background on NC functions, the full Fock space, and the relevant operator algebras generated by the left creation operators. Some preliminary results on the topology of $\scr{O}_d$ and $\scr{O}_d (r)$ are contained in Section \ref{sec:topology}. The main results on the algebraic structure of our algebras and the interplay between the algebraic and the topological structures are in Section \ref{sec:alg_prop}. In Section \ref{sec:usfield} we describe the universal skew fields of NC meromorphic functions, $\scr{M} _d (r)$, more concretely as uniformly analytic NC functions defined on uniformly open NC domains in $r \cdot \rball$ that are analytic-Zariski dense at all sufficiently high levels. Section \ref{sec:application} deals with evaluation in stably finite algebras and our extension of the result of Haagerup and Thorbj{\o}rnsen. Our final section \ref{outlook} describes several open questions and problems.

\section{Preliminaries} \label{sec:prelims}

\subsection{Non-commutative functions} \label{ss:NCfun}

The complex $d$-dimensional, $d\in \N$, NC affine space is 
\[ \ncu := \bigsqcup_{n=0}^{\infty} \cdn; \quad \quad \cdn := \C ^{n \times n} \otimes \C ^{1 \times d}.\] That is, $\ncu$ is the space of all $d-$tuples of complex square matrices of all finite sizes. We will often refer to $\ncu$ as the \emph{non-commutative universe}. We are primarily interested in non-commutative theory and so we typically assume that $d>1$. We choose to view $\cdn$ as row $d-$tuples of $n\times n$ complex matrices, and we will equip $\ncu$ with a `row' operator space structure \cite{Paulsen-cbmaps}.  The classical $d$-dimensional affine space is the collection of all homomorphisms of the polynomial ring in $d$-variables. Similarly, $\cdn$ is the space of all finite-dimensional representations of the free algebra $\fp = \C \langle \fz_1, \cdots, \fz_d \rangle$. It will be convenient and necessary, on occasion,  to include the infinite level, $\ncuext := \ncu \bigsqcup \C ^{(\infty \times \infty) \cdot d}$, where $\C ^{(\infty \times \infty) \cdot d} := \scr{B} (\cH) ^{1 \times d}$ is identified with row $d-$tuples of bounded linear operators on a separable, complex Hilbert space, $\cH$. Choosing $\cH = \ell ^2 (\N)$ and fixing the standard orthonormal basis, $(e _j )_{j=1} ^\infty$, we can identify each level, $\cdn$, of the NC universe as the compression of $\scr{B} ( \ell ^2 (\N) ) \otimes \C ^{1 \times d}=: \scr{B} (\ell ^2 (\N)) ^{1\times d}$ to the subspace spanned by the first $n$ standard basis vectors. That is, we consider $d$-tuples of operators on a separable Hilbert space as the `infinite level'. 

There are two natural operations on the NC universe (that can also be viewed through the lens of representation theory). We have the \textit{direct sum}: For $X \in \cdn$ and $Y \in \cdm$, we define
\[
X \oplus Y = \left( \begin{pmatrix} X_1 & 0 \\ 0 & Y_1 \end{pmatrix}, \cdots, \begin{pmatrix} X_d & 0 \\ 0 & Y_d \end{pmatrix} \right) \in \C ^{((m+n) \times (m+n)) \cdot d}
\]
This operation allows us to connect different levels of our space. The second operation is \textit{joint similarity}. For every $n \in \N$ there is a natural diagonal adjoint action of $\GL_n$ on $\cdn$. Namely, if $S \in \GL_n$ and $X \in \cdn$, $S^{-1} X S := (S^{-1} X_1 S, \cdots, S^{-1} X_d S) \in \cdn$.

A non-commutative (NC) set is any subset $\Omega \subseteq \ncu$ that is closed under direct sums. If $\Omega \subseteq \ncu$ is an NC set, an \textit{NC function} $f \colon \Omega \to \C ^{\N \times \N} = \C^{(\N \times \N) \cdot 1}$ is a function that satisfies the following three natural properties:
\begin{itemize}
    \item $f$ is \textit{graded}: for every $n \in \N$, $f(\Omega_n) \subseteq \C ^{n\times n}$. 

    \item $f$ \textit{respect direct sums}: for every $X, Y \in \Omega$, $f(X \oplus Y) = f(X) \oplus f(Y)$.

    \item $f$ \textit{respects similarities}: for every $n \in \N$, $X \in \Omega_n$, and $S \in \GL_n$, if $S^{-1} X S \in \Omega_n$, then $f(S^{-1} X S) = S^{-1} f(X) S$.
\end{itemize}
It is not hard to check that every free polynomial, $p \in\fp$, defines an NC function on $\ncu$. 

One can topologize $\ncu$ in several ways. The disjoint union topology was first considered by Taylor \cite{Taylor2}, the uniform topology was introduced in \cite{KVV}, and the free topology was studied in \cite{AgMc-freeimp}. It is a surprising result \cite{AgMcC-global, KVV} that for an NC function defined on a disjoint-union open set, local boundedness automatically implies the function is holomorphic. Moreover, any such locally bounded NC function is analytic in the sense that it admits a certain Taylor-type power series expansion, called a \emph{Taylor--Taylor series} about every point in its NC domain with non-zero radius of convergence. In particular, the Taylor--Taylor series of any such function at the origin, $0=(0, \cdots, 0) \in \C^{1\times d}$ of the NC universe is a free formal power series with complex coefficients \cite[Corollary 4.4]{KVV}.

Local boundedness in different topologies on $\ncu$ produces different algebras of NC analytic functions. For our purposes, the uniform topology on the NC universe is most relevant. To introduce the uniform topology, first consider the \emph{NC unit row-ball},
\[
\rball := \bigsqcup_{n=1}^{\infty} \bdn, \ \text{ where } \ \bdn := \left\{ X \in \cdn \left| \ \sum_{j=1}^d X_j X_j^* < I_n \right. \right\}.
\]
That is, $\rball$ is the set of all finite--dimensional strict \emph{row contractions} which means that if $X \in \bdn$, then $\| X = (X_1, \cdots, X_d ) \| _{\scr{B} (\C ^n \otimes \C ^d, \C ^n )} < 1$. Namely, any such $X$ defines a strictly contractive linear map from $d$ copies of $\C ^n$ into one copy. It can also be defined as the open unit ball with respect to the row operator space structure over $\C$ with $d$ components \cite{Paulsen-cbmaps}. For $r >0$, we will denote by $r \cdot \rball$ the appropriate rescaling of $\rball$ with radius $r$. By $\overline{\rball}$ we will denote the level-wise closure of $\rball$ consisting of all finite--dimensional row contractions. For $Y \in \cdn$, also define the \emph{uniform NC unit row-ball centered at} $Y$, $\B ^d _{\N n} (Y)$, as 
\begin{align*} &\B ^d _{\N n} (Y) := \bigsqcup _{m=1} ^\infty \B (Y)^{(mn \times mn)\cdot d}; \\ 
&\B (Y) ^{(mn \times mn)\cdot d}  := \left\{ X \in \C ^{(mn \times mn)\cdot d} \left| \ \| X - I_m \otimes Y \| _{\scr{B} (\C ^{mn} \otimes \C ^d, \C ^{mn})} <1  \right. \right\}.
\end{align*}
Taking $Y=0$ at level one of the NC universe yields the NC unit row-ball. The \emph{uniform topology} on $\ncu$ is then the topology generated by the sub-base of open sets consisting of all uniform NC unit row-balls centered at any point $Y \in \ncu$ \cite[Section 7.2]{KVV}. Again, in \cite[Section 7.2]{KVV}, the uniform topology can be defined with respect to any operator space structure on $\C ^d$; we have chosen to work with the row operator space structure. In this paper, we study algebras of uniformly bounded NC functions defined on uniformly open NC row-balls centered at $0$. Again, since such functions are bounded, hence locally bounded in the uniform topology, they are holomorphic in the sense that they are Fr\'echet differentiable about any point in their uniformly open domains, and they are analytic in the sense that they have a so-called Taylor--Taylor series power series expansion about any point in their domains with non-zero radius of convergence \cite[Chapter 7]{KVV}.  

In \cite[Section 3]{Pop-freeholo}, Popescu introduced the algebra of all complex free formal power series with radius of convergence at least $r \in (0, +\infty]$:
\be \scr{O}_d(r) := \left\{ \left.  f \in \fpsd \right|  \  R_f \geq r \right\}, \ee
where recall that the Cauchy--Hadamard radius of convergence of a free formal power series, $f \in \fps$, also first introduced by Popescu in \cite{Pop-freeholo} is, 
\be \frac{1}{R_f} = \limsup _{\ell \rightarrow \infty} \sqrt[\uproot{3} 2\ell]{\sum _{|\om | = \ell} | \hat{f} _\om | ^2}. \ee

It is not hard to check that each such power series defines an NC function on $r\cdot\rball$ that is bounded on every $s \cdot \overline{\rball}$ for $0 < s < r$. Conversely, if we have an NC function $f$ defined on $\rball$ that is bounded on all row sub-balls of finite radii, then by \cite[Theorem 1.1 and Section 3]{Pop-freeholo}, or by a special case of \cite[Corollary 7.26 and Theorem 8.11]{KVV}, the Taylor--Taylor expansion of $f$ around the origin is a free formal power series that converges absolutely and uniformly on each such ball, and hence defines a uniformly analytic NC function in $r \cdot \rball$. Therefore, by \cite[Theorem 1.1]{Pop-freeholo}, $f \in \scr{O}_d(r)$. In particular, the ring of uniformly entire NC functions, $\scr{O} _d = \scr{O} _d (\infty)$, can be defined as the sub-ring of $\fps$ consisting of free FPS with infinite radius of convergence, or equivalently, as the algebra of NC functions defined on the entire NC universe which are bounded on any NC row-ball of finite radius. It is immediate from the definition that every NC function in $\scr{O}_d$ can be evaluated on any $d$-tuple of linear operators acting on a complex, separable Hilbert space simply by substituting the operators in place of the formal variables in the free FPS expansion.

\begin{remark} \label{Taylorloc}
J.L. Taylor developed a theory of ``localizations'' of the free algebra $\fp$ in \cite{Taylor2}. Taylor's localization is a topological algebra over the free algebra with a property that is today known as a homological ring epimorphism \cite[Definition 4.5]{GeigLenz}. (This is a different notion of ``localization" than we consider in this paper; our universal localizations of semifirs are universal skew fields of fractions.)  One family of such homological ring epimorphisms that Taylor proposes is the family of algebras $S(r)$ for $r = (r_1,\cdots,r_d)$ a $d-$tuple of positive numbers. Here,
\[
S(r) = \left\{ \left. \sum_{\alpha} a_{\alpha} \fz^{\alpha} \right| \text{for all } 1 \leq j \leq d \text{ and } 0 < s_j < r_j,\, \sum_{\alpha} |a_{\alpha}| s^{\alpha} < \infty \right\}.
\]
In this paper, we will show that in many regards the different mode of convergence introduced by Popescu \cite{Pop-freeholo} is more amenable to analysis than the one proposed by Taylor (see also \cite{Luminet-PI} and \cite{KVV}).
\end{remark}

\subsection{Bounded NC functions}

The full Fock space over $\C ^d$ is the Hilbert space $\bH^2_d = \bigoplus_{n=0}^{\infty} \left(\C^d\right)^{\otimes n}$. Equivalently, $\bH^2_d$ is the completion of the free algebra with respect to the inner product that makes the monomials an orthonormal basis.
That is, every element of $\bH^2_d$ can be viewed as a free formal power series with square--summable coefficients and hence $\hardy$ is a natural non-commutative and multivariate extension of the classical Hardy space, $H^2$, of square--summable Taylor series at $0$. (Such power series necessarily have radius of convergence at least $1$ and hence define analytic functions in the complex unit disk.)

The full Fock space is a \emph{non-commutative reproducing kernel Hilbert space} (NC-RKHS) in the NC unit row-ball in the sense of Ball, Marx, and Vinnikov \cite{BMV-kernels}. Namely, for any $Z \in \bdn$, the linear map $f \in \hardy \, \mapsto \, f(Z) \in \C ^{n \times n}$ is a bounded linear map. The NC reproducing kernel of $\hardy$ is a function from $\rball \times \rball$ into completely bounded linear maps and is called the \emph{NC Szeg\"o kernel}: For any $Z \in \bdm$, $W \in \bdn$ and $T \in \C ^{m\times n}$, 
\[
K(Z,W)(T) := \sum_{n=0}^{\infty} \sum_{|\alpha|=n} Z^{\alpha} T W^{\alpha *} \in \C ^{m\times n}; \quad \quad \alpha \in \word.
\]
We will write for every $Z \in \bdn$ and every $y, v \in \C^n$, $K\{Z,y,v\}$ for the element of $\bH^2_d$ that satisfies for every $f \in \bH^2_d$,
\[
\ip{K\{ Z,y,v\}}{f}_{\hardy} = \ip{y}{f(Z)v}_{\C^n} =: y^* f(Z)v.
\]
It is not hard to check that 
\[
K\{Z,y,v\} (\fz) = \sum_{\om \in \word} \ov{y^* Z^{\om} v} \, \fz ^\om.
\]
The vectors $K \{ Z,y,v \}$ are called \emph{NC Szeg\"o kernel vectors}, or \emph{NC point evaluation vectors}.  The left creation operators, $L_i \in \scr{B} (\bH^2_d)$, $1\leq i \leq d$ are the operators of left multiplication by the $d$ independent variables, $L_i = M^L _{\fz _i}$. Similarly, $R_i$ are the operators of right multiplication by the variables. The row $L = (L_1,\cdots, L_d)$ is a \emph{row isometry}, \emph{i.e.} an isometry from $\hardy \otimes \C ^d$ into $\hardy$, which further implies that each $L_i$ is an isometry on $\hardy$ so that the $L_i$ have pairwise orthogonal ranges. The orthogonal complement of the range of $L$ is spanned by the unit vector, $1$, also called the \emph{vacuum vector}. As was observed by Bunce \cite{Bunce}, Frazho \cite{Frazho}, and Popescu \cite{Popescu-dilations,Popescu-vnineq}, $L$ is a non-commutative analogue of the shift on the classical Hardy space of the disk and we will refer to the $L_i$ as the \emph{left free shifts}. Similarly, the \emph{right free shift}, $R = (R_1,\cdots,R_d)$ is a row isometry that is unitarily equivalent to $L$ via the unitary transpose involution that sends the standard basis vector, $\fz ^\alpha$, to $\fz ^{\alpha ^\mrt}$, where if $\alpha = i_1 \cdots i_n \in \word$, then $\alpha ^\mrt := i_n \cdots i_1$. 

The non-commutative Hardy algebra, $\mult$, is the algebra of all uniformly bounded NC functions in the NC unit row-ball, $\rball$. This algebra can be identified with the unital weak operator topology (WOT)--closed algebra generated by $L_1,\cdots,L_d$ \cite{DP-inv}. Moreover, $\bH^{\infty}_d$ can be identified with the algebra of left multipliers of $\bH^2_d$ \cite{Pop-freeholo,SSS}.

Fix $0 < r < 1$, $n \in \N$, $Z \in r \cdot \overline{\bdn}$, and unit vectors $v, y \in \C^n$. Write $\mr{Ad} _Z \colon \C ^{n\times n} \to \C ^{n\times n}$ for the completely positive adjunction map, $\mr{Ad} _Z (T) := \sum_{j=1}^d Z_j T Z_j^*$. By our assumption, 
$\|\mr{Ad} _Z \| = \| \mr{Ad} _Z (I_n) \| \leq r^2$. Then,
\ba
\|K\{Z,y,v\}\|^2 & = & \sum_{n=0}^{\infty} \sum_{|\alpha|=n} \left| y^* Z^\alpha v \right| ^2 \\
& = & \sum_{n=0}^{\infty} y^*  \mr{Ad} _Z ^{\circ n}(vv^*) y \\ 
&\leq & \| y \| ^2 \| v \| ^2 \sum_{n=0}^{\infty} r^{2n} = \frac{\| y \| ^2 \| v \| ^2}{1 - r^2}. \ea 
This simple fact can be used to prove the following lemma, \cite[Lemma 4.5]{JMS-ncBSO}.
\begin{lem} \label{lem:restriction_H^2_d}
Any $f \in \hardy$ is bounded on $r \cdot \overline{\rball}$ for all $0 < r < 1$.
\end{lem}
\begin{proof}
Given any $Z \in r \cdot \overline{\rball}$,
\ba  \|f(Z)\| & = &  \sup \left\{ \left. | y^* f(Z) v |  \ \right| \  \|v\|,\, \|y\| = 1   \right\} \\
& = & \sup \left\{ \left. |\ip{K\{ Z, y ,v \}}{f}| \ \right| \ \|v\|, \, \|y\| = 1 \right\} \\
&\leq & \frac{\|f\|_{\hardy}}{\sqrt{1 - r^2}}. \ea
\end{proof}

For any $f \in \mult$ we define $f(L) := M^L _f \in \scr{B} (\hardy )$ as the bounded linear operator that acts as left multiplication by $f$. We note that for every $0< r < 1$ the operator $rL$ is a strict row contraction. Hence, the bounded linear operator, $f(rL)$, is well-defined. The operator $f(rL)$ corresponds to the restriction of $f$ to $r \cdot \rball$. More generally, this identification extends to $f \in \bH^{\infty}_d \otimes \C ^{n\times k}$, for all sizes of matrices. 

We will say that a function $f \in \bH^{\infty}_d$ and, more generally, $f \in \bH^{\infty}_d \otimes \C ^{n \times k}$ is \emph{inner}, if the operator $f(L)$ is an isometry. Similarly, we say that $f$ is \emph{outer} if $f(L)$ has dense range. Every element $f \in  \mult \otimes \C ^{k \times n}$ admits an \emph{inner--outer factorization}, $f = V \cdot T$, where $V$ is inner and $T$ is outer. This follows from the non-commutative extension of the classical Beurling theorem that was first discovered by Popescu in \cite{Pop-char, Pop-multi} and later proved independently by Davidson and Pitts in \cite{DP-inv}. The proof of this NC Beurling theorem relies on the notion of wandering subspaces for row isometries. Namely, let $\cH \subseteq \hardy \otimes \C^n$ be a closed right shift-invariant subspace. Then, the \emph{wandering subspace} of $\cH$ is defined as the orthogonal complement of the range of $R$ restricted to $\cH \otimes \C^d$,
$$ \scr{W} _\cH := \cH \ominus R (\cH \otimes \C ^d). $$
Elements of $\scr{W} _\cH$ are called \emph{wandering vectors}. We will require several technical lemmas on inner and outer NC functions in the sequel.

\begin{lem} \label{lem:outer_restriction}
Let $F \in \bH^{\infty}_d \otimes \C ^{k \times n}$ be outer. Then, for every $0 < r < 1$, $F(rL)$ is surjective. Moreover, there exists $H_r \in \bH^{\infty} \otimes \C ^{n \times k}$ such that $F(rL) H_r (L)= I_{\hardy} \otimes I_k$.
\end{lem}
\begin{proof}
Since $F$ is outer, $F(L)$ has dense range, by definition, and it is clear that $F(rL)$ will then also have dense range for any $0<r<1$. Moreover, $F(rL)$ will be surjective if and only if $F(rL) ^*$ is bounded below. We first claim that $F(rL)^*$ is bounded below if $F(Z) ^*$ is uniformly bounded below for all $Z \in r \cdot \rball$.

First, since $F$ is outer, we note that $F(Z) : \C ^m \otimes \C ^n \twoheadrightarrow \C ^m \otimes \C ^k$ is surjective onto $\C ^m \otimes \C ^k$ for all $Z \in \cdmr$. Indeed, if there exists a $y \in \C ^m \otimes \C ^k$ and $Z \in \cdmr$ so that $y^* F(Z) =0$, then $K\{Z,y,v\} \in \hardy \otimes \C ^k$ is orthogonal to $\nbran F(L)$ for any $v \in \C ^m$, a contradiction. 

Suppose that $F(rL) ^*$ is not bounded below. Then, given any $\eps >0$ there exists an $x \in \hardy \otimes \C ^k$ so that $\| F(rL) ^* x \| < \frac{\eps}{2} \| x \| $. Since $F(rL) \in \A \otimes \C ^{n \times k}$, it can be approximated in operator norm by matrices of free polynomials. In particular, it is the operator--norm limit of
its Ces\`aro partial sums. This claim can be proven from the properties of the Ces\`aro summation map \cite{DP-inv}, the fact that for a function $f \in \A$, $f(tL)$ converges to $f$ in norm as $t$ increases to $1$ \cite{Popescu-poisson_transform}, a straightforward estimate that shows that the free polynomial partial sums of $f(tL)$ converge in norm to $f(tL)$ for any $t \in (0,1)$, and a standard $\varepsilon/3-$argument. Hence, there exists an $N \in \N$ so that $\ell > N$ implies that $\| F_\ell (rL) ^* - F(rL) ^* \| < \frac{\eps}{2}$, where $F_\ell (rL)$ denotes the $\ell$-th Ces\`aro sum. Finally, we can approximate the ``infinite--dimensional point", $rL$, by finite--dimensional compressions, $Z^{(j)} \in r \cdot \rball$, of size $j$, by compressing $rL$ by the projection, $P_j$ onto the first $j$ standard basis vectors of $\hardy$. Then, since $F_\ell \in \fp \otimes \C ^{k \times n}$, $F _\ell (Z^{(j)}) ^*$ converges, in the strong operator topology, to $F _\ell (rL ) ^*$. Hence, there is a $J \in \N$ so that $j> J$ implies that 
$$ \| F _\ell (Z^{(j)}) ^* P_j x - F(rL) ^*  x \| = \| (F_\ell (Z^{(j)}) ^* - F(rL) ^* ) x \| < \eps \| x \|. $$ Since $P_j x \rightarrow x$ this proves that $F (Z^{(j)}) ^*$ is not uniformly bounded below, as $j \rightarrow \infty$. Hence $F(Z) ^*$ is not uniformly bounded below in $r \cdot \rball$. We conclude that if $F(Z) ^*$ is uniformly bounded below in $r \cdot \rball$, then $F(rL)^*$ is bounded below, and hence $F(rL)$ is surjective. 

We now show that $F(Z) ^*$ is uniformly bounded below in $r \cdot \rball$. If not, then there exist points $W^{(\ell)}$ of size $m_\ell \in \N$ in the row-ball of radius $r$ and unit norm vectors $u^{(\ell)} \in \C ^{m_\ell} \otimes \C ^k$ so that $\| F (W^{(\ell)}) ^* u^{(\ell)} \| \rightarrow 0$.

 Given any $h \in \hardy \otimes \C ^k$, $y = \bsm y_1 \\ \vdots \\ y_i \esm \in \C ^m \otimes \C ^k$, $y_i \in \C ^m$, and $v \in \C ^m$, we define $K\{ W, y, v \} \in \hardy \otimes \C ^k$ by 
$$ \ip{ K \{ W, y, v  \} }{h}_{\hardy \otimes \C ^k} := y^* h(W)v = (y_1 ^*, \cdots, y_k ^*) \bsm h_1 (W) v \\ \vdots \\ h_k (W) v \esm = \sum _{i=1} ^k y_i ^* h_i (W) v. $$ Note that if $y = \sum _{i=1} ^k y_i \otimes e_i$, that 
$$ K\{W, y , v \} = \bpm K \{ W , y_1 , v \} \\ K\{ W, y_2, v \} \\ \vdots \\ K\{ W , y_k , v \} \epm \in \hardy \otimes \C ^k.$$

For any $1 \leq i \leq k$, consider the sequence of vectors 
$$ h_\ell ^{(i)} := K \{ W ^{(\ell)}, u ^{(\ell)}, u ^{(\ell)} _i \}. $$ Here, we write $u^{(\ell)} = \sum_{i=1}^k u^{(\ell)}_i \otimes e_i$.


Then, it is clear that $K \{ W ^{(\ell)}, u ^{(\ell)}, u ^{(\ell)} _i \}$ is a uniformly norm-bounded sequence in $\hardy \otimes \C ^k$ for each $i$ since each $u ^{(\ell)}$ is a unit vector and $W^{(\ell)} \in r \cdot \rball$. It follows that there is a weakly convergent subsequence $h_{\ell_j} ^{(i)}:= K \{ W ^{(\ell_j)}, u ^{(\ell_j)}, u ^{(\ell_j)} _i \}$, with a weak limit, $h ^{(i)} \in \hardy \otimes \C ^k$. Note that 
$$ \hat{h} _\emptyset ^{(i)} = \lim _j \bpm u^{(\ell _j) *} _1 u ^{(\ell_j)} _i \\ \vdots \\ u^{(\ell _j) *} _k u ^{(\ell_j)} _i \epm \in \C ^k, $$ so that 
$$ \| \hat{h} _\emptyset ^{(i)} \| _{\C ^k} \geq \lim _j u^{(\ell _j) *} _i u ^{(\ell_j)} _i = \lim _j \| u ^{(\ell_j)} _i \| ^2. $$ If $\hat{h} _\emptyset ^{(i)} =0$ for every $1 \leq i \leq k$, this would imply that $$ 0 = \lim _j \sum _{i=1} ^k \| u ^{(\ell_j)} _i \| ^2 = \lim _j \| u ^{(\ell _j)} \| ^2 =1, $$ since each $u ^{(\ell_j)}$ is a unit vector. Hence at least one of the $h^{(i)} \neq 0$. 


However, for any $x \in \hardy \otimes \C ^k$, and each $1 \leq i \leq k$,
\ba \ip{h ^{(i)}}{F(L) x}  & = & \lim _j u^{(\ell_j)*} F(W) x(W) u^{(\ell _j)} _i \\
& = & \left( F(W) ^* u^{(\ell_j)} \right)^* x(W) u^{(\ell _j)} _i \rightarrow 0. \ea 
This proves that $F(L)$ does not have a dense range, contradicting that $F$ is outer. 

The last part of the lemma follows from Popescu's corona theorem \cite[Corollary 3.2]{Pop-multi}. Since $T(rL)$ is surjective, $T(rL)^*$ is bounded below, and thus the corona theorem guarantees the existence of a right inverse.
\end{proof}

We note that if $f$ is an NC entire function then the operator $f(rL)$ makes sense for every $r > 0$. Moreover, since the power series of $f$ converges uniformly on any ball, we have $f(rL) \in \A$. Here, $\A$, the \emph{NC} or \emph{free disk algebra} of Popescu is the norm closure of the free algebra in $\bH^{\infty}_d$. Equivalently, $\A$ is completely isometrically isomorphic to the unital norm closed algebra generated by the left creation operators in $\scr{B} (\hardy)$. We will typically view elements of $\A$ and $\mult$ as left multiplication operators on $\hardy$ and given any $F \in \mult \otimes \C ^{m \times n}$ we will sometimes write $F(L)$ in place of $F$, to denote that $F$ acts as left multiplication by $F$. Indeed, evaluating the free formal power series of $F$ at $L=(L_1, \cdots, L_d)$ yields the operator of left multiplication by $F$ (if $F \in \A \otimes \C ^{m \times n}$ the partial Ces\`aro sums of this power series converge in operator-norm, while if $F \in \mult \otimes \C ^{m \times n}$ then these Ces\`aro sums and their adjoints converge in the strong$-*$ operator topology, as was discussed previously). 

\begin{lem} \label{cyclic-wand}
Let $F \in \mult \otimes \C ^{m \times n}$. Then, the dimension of the wandering subspace of $\nbran F(L) ^{- \| \cdot \|}$ is at most $n$.
\end{lem}
\begin{proof}
The wandering space of $\scr{R} (F) := \nbran F(L) ^{-\| \cdot \|}$ is $\scr{W} (F) := \scr{R} (F) \ominus  R \otimes I_n (\scr{R} (F) \otimes \C ^d)$, and the wandering dimension of $\scr{R} (F)$ is the dimension of $\scr{W} (F)$. It is clear that the set $\{ F(L) 1 \otimes e_j \} _{j=1} ^n \subseteq \hardy \otimes \C ^n$ is a cyclic set for the restricted row isometry $R \otimes I_n | _{\scr{R} (F) \otimes \C ^d}$. We claim that if $P$ is the orthogonal projection of $\scr{R} (F)$ onto its wandering subspace, that the set of all $x_j := P (
F(L) 1 \otimes e_j )$, $1\leq j \leq n$ spans $\scr{W} (F)$. If not, suppose that there is an $x \in \scr{W} (F)$ so that $x \perp \{ x_j \} _{j=1} ^n$. Then, since $x \in \scr{W} (F)$, we have that 
$$ x \perp R^\om \otimes I_n x_j, \quad \quad \forall \ \om \neq \emptyset, $$ and $1 \leq j \leq n$, since $R^\om \otimes I_n x_j \in \scr{R} (F) \ominus \scr{W} (F)$ and $x \in \scr{W} (F)$. Moreover, $x \perp x_j$, $1 \leq j \leq n$, by assumption. Hence, 
$$ 0 = \ip{ R^\om \otimes I_n x_j }{x}_{\hardy \otimes \C ^n}, \quad \quad \forall \ \om \in \bW _d. $$ 
Since $\{ x _j \} _{j=1} ^n$ is an $R-$cyclic set, this implies that $x =0$, proving the claim. 
\end{proof}

\begin{lem} \label{lem:inner_restriction}
    Let $F = (f_1, \cdots, f_n) \in \scr{O} _d (r) ^{1\times n}$, $r \in (0, +\infty]$. For every $0<q<r$, let $F(qL) = V_q (L) T_q (L)$ be the inner--outer decomposition. Then, there exists $0< \rho <r$ and $k \leq n$, $k \in \N$, such that for every $q> \rho$, the length of $V_q$ is $k$. Moreover, if $\rho < q < s <r$, then $V_s(\tfrac{q}{s} L)$ is injective and $\nbker F(qL) = \nbker T_s(\tfrac{q}{s} L)$.

    \noindent In particular, if $F(qL)$ is injective, then $T_s(\tfrac{q}{s} L)$ is both injective and surjective, hence square and invertible.
\end{lem}
\begin{proof}
    Let $0<q < s <r$, then
    \[ V_q(L) T_q(L) = F(qL) = V_s(\tfrac{q}{s} L) T_s(\tfrac{q}{s} L).
    \]
    Let $w(t)$ be the length of $V_t$, for $0<t<r$. Then $w(t) = \mr{dim} _{\scr{W}} \nbran F(tL) ^{-\| \cdot \|}$ is the wandering dimension of $\nbran F(tL) ^{-\| \cdot \|}$. We first note that $w(t) \leq n$, by the previous lemma, since $F$ is of length $n$. Moreover, since $T_s(\tfrac{q}{s} L)$ is surjective by Lemma \ref{lem:outer_restriction}, we have that 
$$ \nbran V_q(L) = \nbran F(qL) ^{- \| \cdot \|} = \nbran V_s(\tfrac{q}{s} L) ^{- \| \cdot \|}. $$
    Since the dimension of the wandering subspace is at most the number of components, we obtain that $w(q) \leq w(s)$. In other words, the function $w(t)$ is monotonically increasing, bounded, and integer--valued. Therefore, there exists a $\rho >0$, such that for all $0<\rho <q < r$, $w(q)$ is constant.

    For $q > \rho$, we apply the inner--outer decomposition $V_s(\tfrac{q}{s} L) = V (L) H_{s,q} (L)$, where $H_{s,q}$ is outer and square. Observe that any square outer is necessarily injective. Indeed, at every point $Z \in \rball$, we have that $H_{s,q}(Z)$ is surjective and, therefore, invertible. We conclude that any vector in the kernel of $H_{s,q}$ must evaluate to $0$ on $\rball$, and thus, is zero. Therefore, since $V(L)$ is an isometry, $V_s(\tfrac{q}{s} L)$ is injective.

    The final statement then follows as if $V_s (\tfrac{q}{s} L)$ is injective, $T_s (\tfrac{q}{s} L)$ is surjective and $F(qL) = V_s (\tfrac{q}{s} L ) T_s (\tfrac{q}{s} L)$ is injective, we must also have that $T_s (\tfrac{q}{s} L)$ is injective. By the open mapping theorem, $T_s (\tfrac{q}{s} L)$ is then invertible. Any invertible $A (L) \in \mult \otimes \C ^{m \times n}$ must be square, \emph{i.e.} $m=n$. Indeed, $A(L) : \hardy \otimes \C ^n \rightarrow \hardy \otimes \C ^m$ is surjective. Hence $\{ A(L) 1 \otimes e_k \} _{k=1} ^n$ is an $R \otimes I_m-$cyclic set for $\hardy \otimes \C ^m$ and hence $m \leq n$. Conversely, $A(L) ^{-1} : \hardy \otimes \C ^m \rightarrow \hardy \otimes \C ^n$ is also surjective, so $m=n$.
\end{proof}

\begin{lem}{\cite[Appendix]{JMS-ncBSO}} \label{lem:idempotent}
Let $E \in \mult \otimes \C ^{n \times n}$ be idempotent. Then there exists an $S \in \GL _n ( \mult )$ and a $1 \leq k \leq n$ so that 
$$ E (L) = S(L) \, I_{\hardy} \otimes \bpm I_k & 0 \\ 0 & 0_{n-k} \epm S(L) ^{-1}. $$
\end{lem}
This lemma was established in \cite[Appendix]{JMS-ncBSO}, we include the proof for completeness. 
\begin{proof}
Consider the inner--outer factorization of $E$ and of the complementary idempotent, $I-E$, $E(L) = V(L) F(L)$, $I-E(L) = W(L) G(L)$. Then, 
\ba I_{\hardy} \otimes I_n & = & E (L) + (I-E(L))  \\
& = & V(L) F(L) + W(L) G(L) = (V(L), W(L) ) \bpm F(L) \\ G(L) \epm. \ea 
Write $S(L) := (V(L), W(L) )$. Since $E(L)^2 = E(L)$, we obtain $V(L)(I_{\hardy} \otimes I_k - F(L) V(L)) F(L) = 0$. Since $V(L)$ is injective and $F(L)$ has dense range, it follows that $F(L) V(L) = I_{\hardy} \otimes I_k$. Similarly, $G(L) W(L) = I_{\hardy} \otimes I_j$. Moreover, since $E(L) (I - E(L)) = 0 = (I - E(L)) E(L)$, we conclude $F(L) W(L) = G(L) V(L) = 0$. Hence, $S(L)$ is invertible with inverse $S(L) ^{-1} = \bsm F(L) \\ G(L) \esm$. In particular, both are square, i.e., $j+k = n$.

It is now clear that 
\ba S(L) I_{\hardy} \otimes \bpm I_k & 0 \\ 0 & 0_{n-k} \epm S(L) ^{-1} & = & (V(L), W(L))  I_{\hardy} \otimes \bpm I_k & 0 \\ 0 & 0_{n-k} \epm \bpm F(L) \\ G(L) \epm \\
& = & V(L) F(L) +0 = E(L). \ea 
\end{proof}

\begin{remark}
In \cite{AugMarS}, it was proven that $h \in \scr{O} _d$ if and only if $h$ has a jointly compact and quasinilpotent realization $(A,b,c) \in \scr{B} (\cH ) ^d \times \cH \times \cH$, \emph{i.e.} with $A$ a jointly quasinilpotent $d-$tuple so that each $A_j \in \scr{B} (\cH)$ is compact. This also extends to matrices over $\scr{O} _d$. It readily follows that if $E \in \scr{O} _d ^{n\times n}$ is idempotent so is $I-E$ and then $E \in \mult \otimes \C ^{n \times n}$, so that $E (L) = S(L) I_{\hardy} \otimes I_k S(L) ^{-1}$, where $S(L) ^{-1} = \bsm F(L) \\ G(L) \esm$, and each of $F, G$ have jointly compact and quasinilpotent realizations so that $S ^{-1}$ is also NC entire. It further follows, from the realization algorithm, that $S$ will have a jointly compact realization, but it is not clear, a priori, whether $S$ is entire. 

One of our main results, Theorem \ref{thm:entires_semifir}, proves that each $\scr{O} _d (r)$, $r \in (0, +\infty]$, is a semifir, and hence is projective--free. That is, a ring is projective--free if every finitely generated projective module over the ring is a free module. This implies, in particular, that any idempotent $E \in \scr{O} _d (r) ^{n \times n}$ has such a factorization. Indeed, given such an $E$, we have that $\nbran E$ and $\nbran I-E$ are both finitely--generated projective right $\scr{O} _d (r)-$modules contained in $\scr{O} _d (r) ^n$ obeying $\scr{O} _d (r) ^n = \nbran E + \nbran I-E$ and $\nbran E \cap \nbran I-E = \{ 0 \}$. Since $\nbran E$, $\nbran I-E$ are fintely--generated and $\scr{O} _d (r)$ is semifir, these are both free right $\scr{O} _d (r)-$modules of finite rank. Choosing a basis $\{ f_1, \cdots, f_k \}\subseteq \scr{O} _d (r)$ for $\nbran E$ and a basis $\{ f_{k+1}, \cdots ,f_n \} \subseteq \scr{O} _d (r)$ for $\nbran I-E$, it follows that $\{ f_i \} _{i=1} ^n$ is a basis for $\scr{O} _d (r) ^n$. Defining the matrix $S \in \scr{O} _d (r) ^{n\times n}$ by 
$$ S = \bpm \left. f_1 \right| \cdots \left| f_k \right| \left. f_{k+1} \right| \cdots \left| f_n \right. \epm, $$ \emph{i.e.} if $1 \otimes e_j$ denotes the standard basis of $\scr{O} _d (r) ^n$, then $S 1 \otimes e_j =: f_j$, then $S$ is necessarily bijective, hence invertible by the open mapping theorem. Hence, 
$$ E = S \begin{pmatrix} I_k &0 \\ 0 & 0 \end{pmatrix} S^{-1}, $$ where $S, S^{-1} \in \scr{O} _d (r) ^{n \times n}$. Here, note that $S^{-1}$ maps a basis for $\scr{O} _d (r) ^n$ onto the standard basis so that $S^{-1} \in \scr{O} _d (r) ^{n\times n}$ as well. This applies, in particular, to $\scr{O} _d = \scr{O} _d (+ \infty)$.
\end{remark}

\begin{lem} \label{lem:invertibles_preserve_wan_dim}
Let $S \in \GL_n(\bH^{\infty}_d)$ and let $\hardy \otimes \C^n = \cH \dotplus \cK$ be a decomposition into a non-orthogonal closed direct sum of right invariant subspaces. Then, $\dim \scr{W}_{\cH} + \dim \scr{W}_{\cK} = n$, $\dim \scr{W}_{S\cH} = \dim \scr{W}_{\cH}$, and $\dim \scr{W}_{S\cK} = \dim \scr{W}_{\cK}$.
\end{lem}
In the above statement, we have introduced the notation, $\dotplus$, for a non-orthogonal, topological direct sum of closed subspaces. Namely, given subspaces $\cH, \cJ \subseteq \cK$, we write $\cK = \cH \dotplus \cJ$ if $\cH, \cJ$ are closed subspaces with trivial intersection, $\cH \cap \cJ = \{ 0 \}$, so that $\cK$ is spanned by $\cH$ and $\cJ$, $\cK = \cH + \cJ$.
\begin{proof}
The first claim is proved in \cite[Appendix]{JMS-ncBSO} and in Lemma \ref{lem:idempotent} above. For the second claim, we note that $S \scr{W}_{\cH}$ generates $S \cH$ by right shifts. Therefore, $\dim \scr{W}_{S \cH} \leq \dim \scr{W}_\cH$. Now applying this inequality to $S^{-1}$, we get the equality. Similarly, for $\cK$.
\end{proof}

The following lemma is a variation of \cite[Lemma 5.2]{JMS-ncBSO}.
\begin{lem}
Let $r \in (0,\infty]$ and $F \in \scr{O}_d(r) ^{1\times n}$. Then, $\dim \scr{W}_{\nbker F(qL)}$ is an increasing function of $q \in (0,r)$.
\end{lem}
\begin{proof}
Let $0 < q < s < r$. Then the map $\Phi_{q/s}$ will map $F(sL)$ to $F(qL)$. Therefore, for every $h \in \nbker F(sL)$, we have that $\Phi_{q/s}(h) \in \nbker F(qL)$. Hence, $\cH := \overline{\Phi_{q/s}(\nbker F(sL))}$ is a closed right-invariant subspace of $\nbker F(qL)$. In particular, $\cW := \Phi_{q/s}(\scr{W}_{\nbker F(sL)})$ is a generating set for $\cH$. Let $P$ be the projection onto $\scr{W}_{\cH}$. Let $u \in \scr{W}_{\cH} \ominus P \cW$. Fix an arbitrary $w \in \cW$. Then,
\[ \langle u, R^{\alpha} w \rangle = \langle u, R^{\alpha} P w \rangle + \langle u, R^{\alpha} (1 -P) w \rangle = 0.
\]
Since $\cH \ominus \scr{W}_{\cH}$ is $R-$invariant, the second summand on the left is $0$. Since $R^{\alpha*} u = 0$ for every non-empty $\alpha$ and $u \perp w$ by assumption, the first summand is also $0$. Finally, since $\cW$ is a generating set, $u = 0$, and therefore $\scr{W}_{\cH} = \overline{P \cW}$.
\end{proof}

\begin{remark}
In \cite[Appendix]{JMS-ncBSO}, it is proved that $\bH^{\infty}_d$ is a projective--free ring. However, it is also claimed incorrectly in \cite[Appendix]{JMS-ncBSO} that $\bH^{\infty}_d$ is a semifir. Indeed, this fails even for $d=1$, as can be seen in \cite[Beispiel 2]{vonRenteln}. Namely, if we take $f(z) = e^{\frac{z+1}{z-1}}$ which is a (singular) inner function and $g(z) = 1 -z$, an outer function, then by the corona theorem $f$ and $g$ do not generate the trivial ideal. Moreover, the ideal generated by $f$ and $g$ is not principal. The reason is that if $h = h_i \cdot h_o$ is the inner--outer factorization of an $h$ that divides both $f$ and $g$, then we since $f = h a$, $h_o$ must be invertible since $f$ has no outer part. On the other hand, since $g = h b$, $h_i = 1$ since $g$ has no non-trivial inner part. Hence, $h$ is a unit, but this contradicts the fact that the ideal generated by $f$ and $g$ is not trivial.

As we will see, the obstruction that prevents $\mult$ from being a semifir is that it is generally not possible to ``trivialize" row outers that have dense range in the sense of \cite[Section 2.3 and Theorem 2.3.1]{Cohn}. In this paper, we will see that for outers this problem can be overcome by decreasing the radii of the balls. Namely, any such row outer on a row-ball of a given radius can always be trivialized on any ball of strictly smaller radius. Even though $\mult$ is not a semifir, it would be interesting to determine whether $\bH^{\infty}_d$ is a Sylvester domain, which is a larger class of rings than semifirs that also have universal skew fields of fractions \cite[Section 5.5]{Cohn}.
\end{remark}

\section{Topologies on rings of analytic NC functions} \label{sec:topology}

Classically, the ring of holomorphic functions on a domain in $\C ^d$ has a natural Fr\'echet topology that arises from a compact exhaustion of the domain. Every two such exhaustions give rise to the same topology. Moreover, the resulting topological algebra is locally multiplicatively convex (LMC) \cite{Michael-lmc_book} and nuclear. One can view such algebras as inverse limits of function algebras. The quantized version of a function algebra is an operator algebra and a theory of quantized topological vector spaces was explored by Effros and Webster in \cite{EffWeb,Webster}. Later in Section \ref{sec:application} we will apply this theory of Effros and Webster to define a class of topological ``local operator algebras" which admit well-defined evaluations of the universal skew field, $\scr{M} _d$, of $\scr{O} _d$. 

Our algebras $\scr{O}_d$ and $\scr{O}_d (r)$, $r>0$, admit a natural topology given by the family of semi-norms $$\rho_q(f) = \sup _{0<q <r} \left\{\|f(X)\| \left| X \in q \cdot \overline{\rball} \right. \right\}.$$  It is easily seen that the rings $\scr{O}_d (r)$, $r \in (0, +\infty]$, are complete with respect to this topology. We can view these rings as inverse limits of certain operator algebras. Namely, let us write $\A (r)$ for the closure of the free algebra viewed as functions on $r \cdot \rball$ and equipped with the supremum (matrix) norm topology on this row-ball. (In the case where $r=1$ this recovers the \emph{non-commutative disk algebra}, $\A$, as introduced by Popescu and we will often call $\A (r)$ the free disk algebra on the NC row-ball of radius $r$.) We also let $\mult (r)$ denote the ring of all NC functions that are uniformly bounded in $r \cdot \rball$. This can also be viewed through the lens of operators. For every $r >0$, it is easy to see that $\rho_r(f) = \|f(rL)\|_{\scr{B} (\hardy)}$. Hence, the restriction maps are essentially given by evaluating the free FPS of $f$ at the operator $d-$tuple $rL = (rL_1, \cdots, rL_d)$.

Let $\bH^2_d(r)$ be the $H^2$-space of the ball $r \cdot \rball$. This is defined as the Hilbert space of all free FPS, 
$$ f (\fz )  =  \sum _{\om \in \word} \hat{f} _\om \, \fz ^\om = \sum _{\ell =0} ^\infty \underbrace{\sum _{|\om | = \ell} \hat{f} _\om \, \fz ^\om,}_{=: f _\ell (\fz )} $$ 
with Hilbert space norm given by 
$$ \|f\|_r^2 = \sum_{\ell =0}^{\infty} r^{2\ell} \|f_\ell \|_{\hardy} ^2, $$ and where $\| f _\ell \|_{\hardy}$ is the norm of the $\ell-$th homogeneous component in $\bH^2_d = \hardy (1)$. For any $r>0$, the Hardy algebra, $\mult (r)$, defined above, can be identified with the left multiplier algebra of $\hardy (r)$, and $\A (r)$ can be equivalently defined as the operator--norm closure of the free algebra in $\mult (r)$.  The Hardy algebra, $\mult (r)$ is then the weak operator topology (WOT) closure of $\A (r)$ in $\scr{B} (\hardy (r))$. Also note that for any $0<r<s$, the re-scaling map, $F (Z) \mapsto F(\tfrac{r}{s} Z)$, defines an embedding of $\mult (s)$ into $\mult (r)$ and of $\A (s)$ in $\A (r)$, both with dense ranges. (Albeit, in the former case, in the weak* topology of $\scr{B} (\hardy (r))$.) In particular, $F \in \mult (r)$ if and only if $F(rL) \in \mult = \mult (1)$.

For $r < s$, there are natural restriction maps $\A(s ) \hookrightarrow \A(r)$ and $\bH^2_d(s) \hookrightarrow \bH^2_d(r)$ which we will denote by $\pi _{r,s}$. This gives two families of topological rings, $\{ \A (r) | \ r \in (0, \infty) \}$, and $\{ \hardy (r) | \ r \in (0, \infty ) \}$ equipped with continuous mappings $\pi _{r,s}$, $0<r<s<\infty$, given by restriction to the NC row-ball of radius $r<s$. For convenience, given any $r \in (0, +\infty ]$, we fix a strictly increasing sequence $(r_n)_{n=1} ^\infty$ of positive radii so that $\lim _{n \uparrow \infty} r_n = r$. In the case where $r = +\infty$ we will simply choose $r_n = n \in \N$. We will also write $\pi _{m,n} := \pi _{r_m, r_n} : \A (r_n ) \hookrightarrow \A (r_m)$, $m<n$, for the continuous restriction map. Following \cite{Arens-dense_lim}, the \emph{inverse limit} of $\{ \A (r_n ) \}_{n=1} ^\infty$ is then a topological ring defined as follows. First, consider the Cartesian product space
$$ \prod _{n=1} ^\infty \A (r_n), $$ equipped with the product topology and component-wise multiplication making it a topological ring. The \emph{inverse limit ring}, $\varprojlim \A(r_m)$, of the family $\{ \A (r_n ) \}$, is then defined as the topological sub-ring of $\prod _{n=1} ^\infty \A (r_n)$ consisting of all sequences $(a_n)_{n=1} ^\infty$, $a_n \in \A (r_n)$, obeying $a_n = \pi _{n, n+1} (a_{n+1})$ and equipped with the subspace topology it inherits from the Cartesian product $\prod \A (r_n)$. The inverse limit of the family $\{ \hardy (r_n ) \}$ is defined similarly. Observe that for any $n \in \N$, the restriction map $\pi _{n,n+1} : \A (r_{n+1}) \hookrightarrow \A (r_n)$ is a completely contractive and unital homomorphism of operator algebras with dense range. Indeed, any $a \in \A (r)$ is, by definition, the supremum-norm limit of free polynomials in $r \cdot \rball$. Hence $\fp$ is dense in $\A (r)$ for any $r>0$ which implies that $\pi _{r,s} : \A (s) \hookrightarrow \A (r)$ has dense range for any $0<r<s$. That is, $\varprojlim \A(r_m)$ is a \emph{dense inverse limit ring} in the sense of Arens \cite{Arens-dense_lim}. 

\begin{lem} \label{lem:inverse_limit}
For any $r \in (0,+\infty]$, we have two presentations of $\scr{O}_d (r) $ as inverse limits
\[
\scr{O} _d (r) = \varprojlim_{m \to \infty} \A(r_m) \text{ and } \scr{O} _d (r) = \varprojlim_{m \to \infty} \bH^2_d(r_m).
\]
The maps are the restriction maps defined above, and $0 < r_m < r$ is any strictly increasing sequence with limit $r$.
\end{lem}
\begin{proof}
Consider $r = +\infty$. The first presentation is clear since the topology on $\scr{O}_d$ is given by the supremum norms on balls of increasing radius. For the second one, we note that if a function $f$ extends analytically beyond the ball, then there is a bound on the $\bH^{\infty}_d$ norm by $\bH^2_d$ norm, see Lemma \ref{lem:restriction_H^2_d}. The same proof works for the claim on $\scr{O}_d(r)$, $r>0$.
\end{proof}
Moreover, by the discussion above, we obtain the following: 
\begin{lem}
Given any $r \in (0, +\infty ]$, $\scr{O} _d (r)$ is a complete dense inverse limit topological ring. 
\end{lem}
\begin{proof}
It is easily checked that $\scr{O} _d (r)$ is complete with respect to the topology of uniform convergence on every NC sub-ball, $r' \cdot \rball$, $0<r' <r$.
\end{proof}

Each of the algebras, $\scr{O} _d (r)$, satisfy a version of the classical Montel theorem. Though several Montel-type theorems exist in non-commutative analysis (see, for example \cite[Theorem 13.14]{AgMcY} and \cite[Theorem 5.2]{Pop-freeholo}), we provide a slightly different proof here for the sake of completeness. Recall from \cite[Definition 34.2]{Treves-top_vec_sp} that a topological vector space $\cE$ is called a \emph{Montel space}, if it is Hausdorff, barrelled, and every closed bounded subset is compact. By \cite[Corollary 3 in Section 50]{Treves-top_vec_sp}, every nuclear Fr\'echet space is Montel. In a locally convex space, boundedness is equivalent to boundedness in each of the seminorms defining the topology. In particular, for a Fr\'echet space a bounded sequence has a convergent subsequence by metrizability. Thus, we obtain a template for Montel-type theorems.

\begin{thm} \label{thm:montel}
The space $\scr{O}_d$ is a nuclear Fr\'echet space. In particular, if $f_n \in \scr{O}_d$ is a sequence of entire NC functions that are bounded on NC row-balls of finite radii, then $f_n$ has a convergent subsequence.
\end{thm}
\begin{proof}
By \cite[7.3 Corollary 3]{Schaefer-top_vec_sp}, a Fr\'echet space is nuclear if it is an inverse limit of Hilbert spaces with trace-class connecting maps. In the case of $\scr{O} _d = \varprojlim \bH^2_d(r_m)$, where $r_m >0$ is any strictly increasing sequence that diverges to $+\infty$, these connecting maps are the restriction maps $\pi _{m, m+1} = \pi _{r_m, r_{m+1}}$. Given any $0<r<1$, consider the linear re-scaling map $\Phi _r : \hardy \rightarrow \hardy$ defined by $\Phi _r f = f_r$, $f_r (Z) := f(rZ)$. We note that since $\bH^2_d = \oplus_{k=0}^{\infty} \C _k \langle \fz \rangle$, where $\C _k \langle \fz \rangle$ is the space of homogeneous free polynomials of degree $k$, the map $\Phi_r$ acts as multiplication by $r^k$ on each homogeneous component, $\C _k \langle \fz \rangle$. Hence, $\Phi_r$ is trace class, provided that $r \cdot d <1$, since $\nbdim \C _k \langle \fz \rangle = d^k$. Choose a cofinal system of radii, $r_m$, and hence a cofinal system of semi-norms associated to NC row-balls of those radii, such that the radii grow exponentially, with exponent $>d$. Since $\pi _{m, m+1}$ can be identified with $\Phi _r$, $r := \tfrac{t_m}{t_{m+1}}$, it follows that $\scr{O}_d$ is the inverse limit of a system of Hilbert spaces with trace-class maps, as desired.
\end{proof}

\begin{remark}
    By \cite[Theorem 7.4]{EffWeb}, $\scr{O}_d$ has a unique quantization as a \emph{local operator space}. Later in \ref{sec:application} we will show that since $\scr{O} _d$ is an inverse limit of operator algebras, it can be viewed as a \emph{local operator algebra}.
\end{remark}

One can see here that the fact that we have an infinite radius of convergence for any element of $\scr{O} _d$ is important in the proof of Theorem \ref{thm:montel}. In fact, the algebras $S(r)$ introduced by Taylor (see Remark \ref{Taylorloc}), are not nuclear unless all but at most one of the radii are infinite \cite{Luminet-PI}. It seems that the same reasoning applies in our case of $\scr{O} _d (r)$, $0<r<+\infty$, and so we prove a weaker result for these rings.

\begin{thm} \label{thm:R_montel}
For every $r > 0$, the space $\scr{O}_d(r)$ is a Fr\'echet--Montel space.
\end{thm}
\begin{proof}
By \cite[I.6.2 Theorem]{GelShil-gen_func_vol2}, it suffices to show that if we have a set bounded in the norm induced from $\bH_2^d(r_{m+1})$, then it contains a Cauchy subsequence in the norm from $\bH^2_d(r_m)$. However, this follows immediately from the fact that the restriction map from $\bH_2^d(r_{m+1})$ to $\bH^2_d(r_m)$ is compact. Indeed, setting $t:= \tfrac{r_m}{r_{m+1}} <1$, we can identify the restriction map with the linear re-scaling map, $\Phi _t \in \scr{B} (\hardy)$ defined by $\Phi _t (f) := f(t L) 1$. It is clear that this linear map is diagonal with respect to the standard orthonormal basis of monomials for $\hardy$, and $\Phi _t (\fz ^\om) = t^{|\om |} \fz ^\om$. Hence the diagonal entries of $\Phi _t$ converge to $0$ as $|\om | \rightarrow \infty$ so that $\Phi _t$, and hence the restriction map, are compact. Since a bounded subset in $\bH_2^d(r_{m+1})$ is weakly precompact, its image in $\bH_2^d(r_{m})$ (under the restriction map, $\pi _{m, m+1}$) is pre-compact in norm.
\end{proof}

The following lemma will be useful in the next section. It describes the density of the image of invertible elements under the restriction map.

\begin{lem} \label{lem:invertible_dense}
    Let $0 < r < s$. For every $n \in \N$, the image of $\GL_n(\cO_d)$ in $\GL_n(\A(r))$ is dense. In particular, the image of $\GL_n(\A(s))$ in $\GL_n(\A(r))$ is dense.
\end{lem}
\begin{proof}
    Let us fix $F \in \GL_n(\A(r))$. Let us consider the norm continuous path $\gamma \colon [0,1] \to \GL_n(\A(r))$ given by $\gamma(t)(Z) = F((1-t)Z)$, for every $Z \in r \cdot \rball$. Clearly, $\gamma$ connects $F$ to the constant function $F(0)$. Using the fact that $\GL_n(\C)$ is path connected, we connect $F(0)$ to the identity matrix. By a known characterization of the path-connected component to the identity in the group of invertible elements of a unital Banach algebra, there exist $G_1,\ldots,G_m \in M_n(\A(r))$, such that $F = e^{G_1} \cdot e^{G_2} \cdots e^{G_m}$. (See, for example,  \cite[Theorem 2.14]{Douglas-banach_alg_tech}.)
    Using Ces\`aro sums (as discussed in the proof of Lemma~\ref{lem:outer_restriction}), we find sequences of matrices of polynomials $Q_k^{(j)} \rightarrow G_j$ in $\A(r)$. It is straightforward to check that the exponential map on a Banach algebra is continuous. (This can be done directly, or by applying upper semi-continuity of the spectrum and the Riesz--Dunford holomorphic functional calculus.) Hence, the sequence of NC entires $H_k = e^{Q_k^{(1)}} \cdot e^{Q_k^{(2)}} \cdots e^{Q_k^{(m)}}$ converges to $F$ in $\GL_n(\A(r))$.
\end{proof}

Classically, the ring of entire functions is a B\'ezout domain. (And a commutative ring is a semifir if and only if it is a B\'ezout domain.) Moreover, every finitely generated ideal of the ring of entire functions (and even the ring of holomorphic functions on a domain in $\C$) is principal and closed. This result is called Helmer's theorem according to \cite[Theorem 13.6]{LueRub} and the converse also holds. Namely, if an ideal of the ring of entire functions is closed, then it is principal. The key tool in the proof of those results is the Weierstrass product theorem that allows one to build entire functions with prescribed zeroes and multiplicities. In other words, closed ideals in the ring of entire functions are described by analytic varieties (counting the multiplicity). Varieties of various types have been considered in the non-commutative setting \cite{HKV-poly,JMS-ncBSO,SSS}. In what follows we consider an ``NC directional variety", in the sense of G. M. Bergman, see \cite{HelMcC-positivss}. Given $S \subseteq \scr{O}_d$, define the corresponding variety as in \cite{JMS-ncBSO},
\[
\scr{Z}(S) = \bigsqcup_{n \leq \aleph_0} \left\{ \left. (X,y) \in \cdn \times (\C ^n \sm \{ 0 \} )  \right| \ y^* h(X) =0 \ \, \forall h \in S  \right\}.
\]
Here, $\C^{\aleph_0} = \ell^2 (\N)$ and $\C ^{(\infty \times \infty)\cdot d} := \scr{B} (\ell ^2)$. Clearly, if $I$ is the closed right ideal generated by $S$, then for every $g \in I$ and every $(X,y) \in \scr{Z} (S)$, we have that $y \perp \nbran g(X)$. Therefore, the variety that we have defined corresponds to closed right ideals. Conversely, given $\scr{Z} \subseteq \bigsqcup_{n\leq \aleph_0} \cdn \times (\C^n \setminus \{0\})$, we can define the corresponding closed right ideal of entire functions. Namely,
\[
I(\scr{Z}) = \left\{f \in \scr{O}_d \left| \  y \perp \nbran f(X) \ \, \forall (X,y) \in \scr{Z} \right. \right\}.
\]
We can also define corresponding varieties in the case of $\scr{O}_d(R)$, $R>0$.

\begin{lem} \label{lem:approx_nullstellensatz}
If $J \vartriangleleft \scr{O}_d (r)$, $r \in (0, +\infty]$ is a right ideal, then $I(\scr{Z}(J)) = \overline{J}$. 
\end{lem}
\begin{proof}
Let $g \notin \overline{J}$. Then, there exists a continuous functional $\varphi$ on $\scr{O}_d (r)$, such that $\varphi(g) \neq 0$, but $\varphi|_{\overline{J}} = 0$. Since $\scr{O}_d(r)$ is an inverse limit, there exists $m \in \N$, such that $\varphi$ factors through $\hardy (r_m)$. Therefore, there is an $h \in \bH^2_d(r_m)$, such that for all $f \in J$ and every free polynomial $p$, $\ip{h}{M^R _p f}_{\hardy (r_m)} = 0$ and hence $h$ is orthogonal to the range of $f(r_m L)$. However, by our assumption, $\ip{h}{g}_{\hardy (r_m)} \neq 0$ so that $g \notin I(\scr{Z}(J))$. 
\end{proof}

The above lemma will later be used to prove a version of the Bergman Nullstellensatz for each $\scr{O} _d (r)$, $r \in (0, +\infty]$.

\section{Algebraic properties of rings of analytic NC functions} \label{sec:alg_prop}

In this section, we will prove two main results. First, we will prove that each $\scr{O}_d (r)$ is a semifir, and secondly we will show that every finitely--generated right ideal in $\scr{O}_d (r)$ is closed. From the latter fact, we derive a version of the Bergman Nullstellensatz for each $\scr{O} _d (r)$, $r \in (0, +\infty]$. Recall from \cite{Cohn} that a ring, $\scr{R}$, is called a \emph{semifir} or \emph{semi-free ideal ring} if every finitely generated right (equivalently, left) ideal is free as a right (respectively, left) module. Alternative conditions for $\scr{R}$ to be a semifir are provided in \cite[Theorem 2.3.1]{Cohn}. In particular, we will employ the condition \cite[Theorem 2.3.1(a)]{Cohn}. In order to state this condition, recall that for two $n$-tuples of elements $x_1,\cdots,x_n,y_1,\cdots,y_n \in \scr{R}$, if $\sum_{j=1}^n x_j y_j =0$, this equation is said to be an \emph{$n-$term relation for $\scr{R}$}. Such a relation is said to be \emph{trivial}, if for every $1 \leq i \leq n$, either $x_i = 0$ or $y_i = 0$. Write $x = \bsm x_1 \\ \vdots \\ x_n \esm \in \scr{R} ^n$ and $y = \bsm y_1 \\ \vdots \\ y_n \esm \in \scr{R} ^n$. We say that an invertible matrix $S \in \scr{R} ^{n\times n}$ \emph{trivializes} the relation $x ^\mrt y = 0$, $x^\mrt := (x_1, \cdots, x_n ) \in \scr{R} ^{1\times n}$, if the new relation $x^{';\mrt} y' := (x ^\mrt S) (S^{-1} y) = 0$ is trivial. We can now state the condition \cite[Theorem 2.3.1(a)]{Cohn}: A ring, $\scr{R}$, is a semifir if and only if every finite relation, $\sum_{j=1}^n x_j y_j =0$, $n \in \N$, can be trivialized by an invertible matrix in $\scr{R} ^{n\times n}$. Since our goal is to prove that $\scr{O}_d (r)$ is a semifir, for any $0<r \leq +\infty$, we will fix an arbitrary $n \in \N$ and elements $f_1,\cdots, f_n, g_1, \cdots, g_n \in \scr{O}_d (r)$, such that $\sum_{j=1}^n f_j g_j = 0$. We will write 
\[
F = \begin{pmatrix} f_1 & \cdots & f_n \end{pmatrix} \in \scr{O}_d (r) ^{1\times n} \quad \text{and} \quad G = \begin{pmatrix} g_1 \\ \vdots \\ g_n \end{pmatrix} \in \scr{O}_d (r)^n,
\]
so that our relation is $F \cdot G = 0$. Our basic strategy will be to first show that given any such $n-$term relation for $\scr{O} _d (r)$, it can be trivialized in $\A (q)$ for any $0<q<r$. We then construct an inverse limit of invertible $n\times n$ matrices over $\A (q)$, $0<q<r$, which trivialize our relation in $q \cdot \rball$, \emph{i.e.} in $\A (q)$, to obtain an invertible matrix over $\scr{O} _d (r)$ that trivializes our original relation for the inverse limit ring, $\scr{O} _d (r)$.

For the remainder of this section, fix any $0<r\leq +\infty$, and a strictly increasing sequence of positive numbers, $(t_j)_{j=1} ^{\infty}$, so that $t_n$ converges to $r$. In the case where $r=+\infty$, $t_j \uparrow +\infty$ diverges to $+\infty$ and we will choose $t_j =j \in \N$.

\begin{lem} \label{lem:all_solutions}
Let $\Phi \in \mult \otimes \C ^{1\times n}$ . Then, there exists $S \in \GL_n(\mult)$, such that $\Phi S = (\tilde{\varphi}_1,\cdots,\tilde{\varphi}_\ell ,0,\cdots,0)$, where $\tilde{\Phi} := (\tilde{\varphi}_1,\cdots, \tilde{\varphi}_\ell)$ is injective, if and only if there exists a (topological) direct sum decomposition $\bH^2_d \otimes \C^n = \cH \dotplus \nbker \Phi$, where $\cH$ is a closed right shift-invariant subspace. Moreover, given any other $T \in \GL_n(\bH^{\infty}_d)$, so that $\Phi T$ has a decomposition with its first $\ell$-coordinates an injective row and where the final $n-\ell$ components are $0$, we have that $T$ satisfies the same properties as $S$
\[
    T = S \left(\begin{smallmatrix} A & 0_{\vphantom{W}} \\ B & C \end{smallmatrix}\right); \quad A \in \mult \otimes \C ^{(n-\ell) \times (n-\ell)}, \ B \in \mult \otimes \C ^{\ell \times (n-\ell)}, \ C \in \mult \otimes \C ^{\ell \times \ell}
\]
\end{lem}

\begin{proof}
If there is such a decomposition, then we can choose $S = (V_1, V_2)$, where $V_1$ is the inner with range $\cH$ and $V_2$ is the inner with range $\nbker \Phi$. Conversely, if there is an $S$, such that $\Phi S$ has the form above, then set $\cW = \bigvee \{e_{n-\ell +1}, \cdots, e_n\} \otimes \bH^2_d$ and note that $S^{-1} \nbker \Phi = \cW$. Therefore, if $\cH = S \cW ^{\perp}$, then $\cH \dotplus \nbker \Phi = \bH^2_d \otimes \C^n$. If $T$ is another such invertible matrix then $T^{-1}S \cW = \cW$. Hence, $T^{-1}S$ is block lower-triangular and by a symmetric argument, so is $S^{-1} T$. 
\end{proof}

\begin{lem} \label{lem:local_trivialization}
Given $F \in \scr{O} _d (r) ^{1\times n}$, $G \in \scr{O} _d (r) ^n$ so that $F\cdot G \equiv 0$, then for any $0<q<r$, there exists $S \in \GL_n( \A (q))$, such that 
\[
FS = \begin{pmatrix} \underbrace{\tilde{F}}_{1 \times (n-\ell)} & \underbrace{0}_{1 \times \ell} \end{pmatrix} \text{ and } S^{-1} G = \begin{pmatrix} \underbrace{0}_{(n-\ell) \times 1} \\ \underbrace{\tilde{G}}_{1 \times \ell} \end{pmatrix}, 
\]
in the row-ball of radius $q>0$. In other words, we can trivialize the relation $FG = 0$ in $\A (q)$. Moreover, there exists $0<r_0<r$, such that for $q > r_0$, $\tilde{F}$, is injective and $\ell$ is independent of $q \in (r_0, r)$.
\end{lem}
\begin{proof}
Choose any $q,s,t >0$ so that $0<q<s<t <r$. Factor $F(tL) = V (L) T (L)$, where $V \in \mult \otimes \C ^{1\times k}$ is inner and $T \in \mult \otimes \C ^{k \times n}$ is outer ($k \leq n$). Now consider $T(\tfrac{s}{t} L)$ (which corresponds to the restriction of $T$ to the ball of smaller radius $s>0$). By Lemma \ref{lem:outer_restriction}, there exists an $H_s \in \mult \otimes \C ^{n\times k}$, such that $T(\tfrac{s}{t} L) H_s (L) = I _{\hardy} \otimes I_k$. In particular, we have a (generally) non-orthogonal, topological direct sum, $\hardy \otimes \C^n = \nbran H_s (L) \dotplus \nbker T (\tfrac{s}{t} L)$. That is $\nbran H_s (L)$ is a closed subspace, since $H_s (L)$ is bounded below, $\nbker T (\tfrac{s}{t} L)$ is also closed, and these two closed subspaces have trivial intersection. Recall that we use the notation $\dotplus$ to denote such a topological, and not necessarily orthogonal direct sum.

Let $\Theta \colon \hardy  \otimes \C^{n-\ell} \to \hardy  \otimes \C^n $ and $\Xi \colon \hardy  \otimes \C^\ell \to \hardy  \otimes \C^n$ be inners, such that $\nbran \Theta = \nbran H_s (L)$ and $\nbran \Xi = \nbker T (\tfrac{s}{t} L)$. (Both $\nbran H_s (L)$ and $\nbker T (\tfrac{s}{t}L)$ are closed $R \otimes I_n-$invariant subspaces so that the NC Beurling theorem applies.) Let $S (L) = ( \Theta (L),  \Xi (L) ) \in \GL_n (\mult  )$. Then, $S$ is invertible (since the Friedrichs angle between the spaces is not $0$). Alternatively, since $\hardy \otimes \C ^n = \nbran \Theta (L) \dotplus \nbran \Xi (L)$ is a topological direct sum, there is a bounded, idempotent operator, $E \in \scr{B} (\hardy \otimes \C ^n)$ so that $\nbran E = \nbran \Theta (L)$ and $\nbran I-E = \nbran \Xi (L)$. Then, for any $x \in \hardy \otimes \C ^n$ and $1 \leq j \leq d$,
$$ (R_j \otimes I_n) E x  =  (R_j \otimes I_n) E^2 x = E (R_j \otimes I_n) Ex, $$ since $E$ is idempotent and $\nbran E = \nbran \Theta (L)$ is $R\otimes I_n-$invariant. Similarly, 
\ba E (R_j \otimes I_n) x & = & E (R_j \otimes I_n ) (E + I-E)x \\
& = & E (R_j \otimes I_n) E x = (R_j \otimes I_n) E x, \ea
and we conclude that $[R_i \otimes I_n, E ]=0$. By Davidson--Pitts \cite[Theorem 1.2]{DP-inv}, we conclude that $E = E(L) \in \mult \otimes \C ^{n\times n}$, and Lemma \ref{lem:idempotent} now implies that $S(L) = (\Theta (L), \Xi (L))$ is invertible. We now have that,
\[
T(\tfrac{s}{t} L) S (L) = \begin{pmatrix} \underbrace{\star}_{k \times (n - \ell)} & \underbrace{0}_{k \times \ell} \end{pmatrix}.
\]

Restricting further to the ball of radius $q>0$, the above identity still holds for $T(\tfrac{q}{t}L) S (\tfrac{q}{t} L)$. Since $0 = F(tL) G(tL) = V (L) T (L) G(tL)$, it follows that $T (L) G(tL) = 0$ and hence $T(\tfrac{s}{t} L) G(sL) = 0$. That is, the range of $G(sL)$ is contained in $\nbker T(\tfrac{s}{t} L) = \nbran \Xi (L) $. We conclude that $G(sL) = \Xi (L) G' (L)$, for some matrix multiplier $G' (L)$ of the correct size. Therefore, 
\ba  S (L) ^{-1} G(sL) & = & S(L) ^{-1} \Xi (L) G' (L) \\
& = & S(L) ^{-1} (0 _{n \times (n-\ell)}, \Xi (L) ) G' (L) \\
& = & S(L) ^{-1} S(L) \bpm 0 _{n-\ell} & 0 \\ 0 & I _{\ell} \epm G' (L) \\
& = & \bpm \underbrace{0}_{(n-\ell) \times 1} \\ \underbrace{\star}_{\ell \times 1} \epm.
\ea
Clearly, this relation holds in $\A (q)$, \emph{i.e.} for $S(\tfrac{q}{s} L) ^{-1} G(qL)$. Lastly,
\ba F(qL) S(\tfrac{q}{s} L) & = & F \left( t (\tfrac{q}{t} L) \right) S (\tfrac{q}{s} L) \\
& = & V (\tfrac{q}{t} L) T ( \tfrac{q}{t} L)   S (\tfrac{q}{s} L) \\
& = & V (\tfrac{q}{t} L) T \left( \tfrac{s}{t} (\tfrac{q}{s} L) \right) \left( \Theta (\tfrac{q}{s} L), \Xi (\tfrac{q}{s} L) \right). \ea

Since $\nbran \Xi (L)  = \nbker T (\tfrac{s}{t} L)$ it follows that $\nbran \Xi (\tfrac{q}{s} L) \subseteq \nbker T (\tfrac{q}{t} L)$ so that the above equation becomes
$$ F(rL) S(\tfrac{q}{s} L) = V(\tfrac{q}{t}L) \begin{pmatrix} \underbrace{\star}_{k \times (n - \ell)} & \underbrace{0}_{k \times \ell} \end{pmatrix} = \begin{pmatrix} \underbrace{\star}_{1 \times (n-\ell)} & \underbrace{0}_{1 \times \ell} \end{pmatrix}. $$
In conclusion, we have trivialized our relation in $\A (q)$. Namely, if we restrict $F$ to $q \cdot \rball$, then for any $Z \in q \cdot \rball$, $F(Z) = F(q (\tfrac{1}{q} Z))$, where $\tfrac{1}{q} Z \in \rball$. It then follows, by the above equation, that 
$$  F(Z) S(\tfrac{1}{s} Z) = V(\tfrac{1}{t}Z) \begin{pmatrix} \underbrace{\star}_{k \times (n - \ell)} & \underbrace{0}_{k \times \ell} \end{pmatrix} = \begin{pmatrix} \underbrace{\star}_{1 \times (n-\ell)} & \underbrace{0}_{1 \times \ell} \end{pmatrix}, $$ for any $Z \in q \cdot \rball$.

Assume now that $q > \rho$, where $\rho$ is the constant obtained in Lemma \ref{lem:inner_restriction}. By Lemma \ref{lem:inner_restriction}, $\nbker F(qL) = \nbker T(\frac{q}{t}L) $ and, by our construction, $T(\frac{q}{t}L) S(\frac{q}{s} L) = ( \tilde{T}, 0 )$. Note that $\tilde{T}$ must be injective and, thus, so is $\tilde{F} = V(\frac{q}{t} L) \tilde{T}$.

It remains to show that for $0<q<r$ sufficiently large, $\ell$, is independent of $q$. Let us write $\ell_q := \ell$ to emphasize this dependence. For $q > \rho$, we have that $\ell_q = \dim _\scr{W} \, \nbker F(qL)$, the wandering dimension of $\nbker F(qL)$. Clearly, for $r' > q > \rho$, the restriction map induces an injective map from $\nbker F(r' L)$ to $\nbker F(q L)$. The closure of the image $\nbker F(r' L)$ is a closed right invariant subspace of $\nbker F(q L)$. Hence, $\ell_{r'} \leq \ell_q$. Namely, the function $q \mapsto \ell_q$ is monotonically decreasing, integer-valued, and bounded below by $1$. Hence, for $q$ big enough, it is constant and we choose $0< \rho \leq r_0 <r$ so that $\ell _q$ is constant for all $0<r_0 <q <r$.
\end{proof}

The proof of the following theorem is inspired by the inverse limit results of \cite{Arens-dense_lim}.

\begin{thm} \label{thm:entires_semifir}
The ring of uniformly analytic NC functions, with radius of convergence at least $r \in (0, +\infty]$, $\scr{O}_d (r)$, $d \in \N$, is a semi-free ideal ring. In particular, the ring of uniformly entire NC functions, $\scr{O} _d = \scr{O} _d (\infty)$, is a semifir.
\end{thm}
\begin{proof}
Let $F$ and $G$ be as above, and let $r_0$ be the number obtained in the preceding lemma so that for any $q \geq r_0$, 
the wandering dimension of $\nbker F(qL)$ is constant and equal to $\ell \in \N$, $\ell < n$.
Given the strictly increasing sequence $(t_m)_{m=1} ^\infty$, $0<t_m < r$, $t_m \uparrow r$, fix $m \in \N$, such that $t_m > r_0$. By Lemma \ref{lem:local_trivialization}, there exist $S \in \GL_n(\A (t_{m+1}))$ and $S' \in \GL_n ( \A (t_m) )$ be such that
\[
F S = \bpm \tilde{F}  & 0 \epm, \ S^{-1} G = \bpm 0 \\ \tilde{G} \epm, \quad \mbox{and} \quad F S' = \bpm \tilde{F}' & 0 \epm, \ (S')^{-1} G = \bpm 0 \\ \tilde{G}' \epm,
\]
in the sense that the first two equations hold for all evaluations in $t_{m+1} \cdot \rball$ and the second two hold for all evaluations in $t_m \cdot \rball$, $\wt{F} \in \A (t_{m+1}) \otimes \C ^{1\times (n-\ell)}$, $\wt{G} \in \A (t_{m+1}) \otimes \C ^{\ell}$, $\wt{F} ' \in \A (t_m) \otimes \C ^{1 \times (n-\ell)}$ and $\wt{G} ' \in \A (t_m) \otimes \C^{\ell}$. Moreover, by the aforementioned lemma, $\tilde{F}$ and $\tilde{F'}$ are both injective. 

Recall that $\pi_{m+1,m} : \A (t_{m+1}) \hookrightarrow \A (t_m)$ denotes the restriction map. By our assumption, $\pi_{m+1,m}(S)$ maps $\cW = \mathrm{Span}\{e_{n-\ell +1}, \cdots, e_n\} \otimes \bH^2( t_m \cdot \rball)$ into the kernel of $F$ so that $\pi_{m+1,m}(\tilde{F})$ must be injective. To be more precise, if it was not injective, then $\nbker F (t_m L)$ would be a right invariant closed subspace that would properly contain $\cW$ and therefore it would have more wandering vectors than $\ell$. However, the wandering dimension of $\nbker F(t_m L)$, is $\ell$ by the previous lemma, since we assume $t_m > r_0$. Let $m_0 \in \N$ be the smallest natural number so that $t_{m_0} > r_0$. 

We conclude, from Lemma \ref{lem:all_solutions}, that for any $m\geq m_0$, there exists a block lower triangular $T' = \bsm A' & 0 \\ B' & C' \esm \in \GL_n( \A (t_m))$, with $A' \in \mult \otimes \C ^{(n-\ell) \times (n-\ell)}$, $C' \in \mult \otimes \C ^{\ell \times \ell}$, such that $S' = \pi_{m+1,m}(S) T'$. For every $\varepsilon > 0$, we can choose a block lower triangular $T \in \GL_n( \A ( t_{m+1}))$ so that $\|\pi_{m+1,m}(T) - T'\| < \frac{\varepsilon}{\|S\|}$. Indeed, 
$$ S' = \pi _{m+1, m} (S) T' \quad \Rightarrow \quad S' T^{';-1} = \pi _{m+1,m} (S). $$ Since both $S'$ and $\pi _{m+1,m} (S)$ trivialize our relation $F \cdot G=0$ on the row-ball of radius $t_m$, Lemma  \ref{lem:all_solutions}
implies that there is a block lower triangular $\wt{T} '$ so that $\pi _{m+1,m} (S) = S' \wt{T}'$. It follows that $\wt{T}' = T^{';-1}$, so that $T^{';-1}$ is also block lower triangular (with the same sized blocks). If 
$$ T' =: \bpm A' & 0 \\ B' & C' \epm, $$ the fact that $T^{';-1}$ is also block lower triangular forces 
the diagonal blocks, $A'$ and $C'$ to be invertible, so that 
$$ T^{';-1}  = \bpm A^{';-1} & 0 \\ -C^{';-1} B' A^{';-1} & C ^{';-1} \epm. $$ Note that the images of $\A(t_{m+1})^{n\times n}$ and $\GL_n (\A(t_{m+1}))$ in $\A(t_{m})^{n\times n}$ and $\GL_n(\A(t_m))$, respectively, are dense. The latter statement was proved in Lemma \ref{lem:invertible_dense}. Hence, there exist matrices $A,B,C$ over $\A (t_{m+1})$ of appropriate sizes so that 
$A,C$ are invertible and so that their images under the restriction map, $\pi _{m+1,m}$ are arbitrarily close to $A',B',C'$, respectively. That is, for any $\eps >0$, we can choose 
$$ T := \bpm A & 0 \\ B & C \epm  \in \GL _n (\A (t_{m+1})), $$ so that $\|\pi_{m+1,m}(T) - T'\| < \frac{\varepsilon}{\|S\|}$, and hence,
\[ \|\pi^{(n)}_{m+1,m}(S T) - S'\| \leq \|\pi^{(n)}_{m+1,m}(S)\| \|\pi_{m+1,m}(T) - T'\| < \varepsilon.
\]
In other words, we can replace $S$ with a different trivializing matrix, such that its image will be close to the solution $T$ within the desired arbitrary precision in the row-ball of radius $t_m$.

Given a complex, separable Hilbert space, $\cH$ and an invertible $B \in \scr{B} (\cH)$, let $r(B) := \|B^{-1}\|^{-1}$. It is well known that the open ball of radius $r(B)$ around $B$ is contained in the group of invertible elements of $\scr{B} (\cH)$. If $B \in \GL_n (\A (t_m))$, we will write $r_m (B)$ in place of $r(B)$ in what follows.

Using the above observations, we will inductively build sequences, $\varepsilon_m > 0$, and $S_m \in \GL_n(\A (t _{m + m_0}))$, for all $m \in \N$, with the following properties:
\begin{itemize}[nosep]
    \item[(i)] For all $m \in \N$, $F_m S_m = (\tilde{F}_m, 0_{\ell} )$ with $\tilde{F}_m$ injective, and where $F_m$ denotes the restriction of $F$ to the row-ball of radius $t_m$,  
    \item[(ii)] For all $m \in \N$, $\|\pi^{(n)}_{m+m_0+1,m+m_0}(S_{m+1}) - S_m\| < \varepsilon_m$,
    \item[(iii)] $\sum_{m =1}^{\infty} \varepsilon_m < \infty$,
    \item[(iv)] For every $m \in \N$ and every $k \in \N$, $\varepsilon_{m+k} \leq r_{m+m_0}(S_m)/2^{m+k}$.
\end{itemize}
For each $m > m_0$, by Lemma \ref{lem:local_trivialization}, we can find $T_m \in \GL_n( \A (t_m))$, such that $\pi_{m}(F) T_m = (\tilde{F}_m, 0_{\ell} )$, with $\tilde{F}'_{m}$ injective. For $m = 1$, we set $S_m = T_{m_0 + 1}$ and let $\varepsilon_1 = \min\{1, r_{m_0 +1}(S_1)/2\}$. Assume that $S_1,\cdots,S_m$ and $\varepsilon_1,\cdots,\varepsilon_m$ have been defined. By the argument above, an $S_{m+1} \in \GL_n(\A(m_0 + m +1))$ exists, such that $\|\pi^{(n)}_{m_0 + m +1,m_0 + m}(S_{m+1}) - S_m\| < \varepsilon_m$ and $\pi_{m_0 + m +1}(F) S_{m+1} = (\tilde{F}_m, 0_{\ell} )$, with $\tilde{F}_{m+1}$ injective. We now set 
\[
\varepsilon_{m+1} = \frac{1}{2^m} \min\{1, r_{m_0+1}(S_1)/2,\cdots, r_{m_0 + m +1}(S_{m+1})/2\}.
\]
The first and second properties are satisfied immediately from the construction. To check the last property, let us fix an arbitrary $m \in \N$. By construction, $\varepsilon_{m+k}$ is a minimum taken over a finite set of numbers, one of which is $r_{m_0+m}(S_m)/2$ divided by $2^{m+k-1}$. 

Having these sequences at hand, we fix $m \in \N$ and set $S_{m,0}' := S_m$. For every $k \in \N$, we further set $S_{m,k}' := \pi_{m+m_0+k,m+m_0}^{(n)}(S_{m+k})$. Since every $\pi_{m_0+m+1,m_0+m}$ is a complete contraction, so is $\pi_{m_0+m+k,m_0+m}$. We note now that
\ba
\|S_{m,k+1}' - S_{m,k}' \| & = & \|\pi_{m_0+m+k+1,m_0+m}^{(n)}(S_{m+k+1}) - \pi_{m_0+m+k,m+m_0}^{(n)}(S_{m+k})\|  \\
&\leq &  \|\pi_{m_0+m+k+1,m_0+m+k}^{(m)}(S_{m+k+1}) - S_{m+k}\| < \varepsilon_{m+k}.
\ea
A standard argument now shows that $S_{m,k} '$ is a Cauchy sequence. Moreover, 
\ba
\|S_{m,k}' - S_m\| & \leq & \sum_{j=1}^k \|S_{m,k-j+1} - S_{m,k-j}\| < \sum_{j=0}^{k-1} \varepsilon_{m+j} \\
& \leq & \dfrac{r_{m+m_0}(S_m)}{2^m} \sum_{j=0}^{k-1} \dfrac{1}{2^j} \\
& < & \dfrac{r_{m+m_0}(S_m)}{2^m}. \ea 

In particular, the Cauchy sequence stays in a closed ball of radius $\frac{r_{m+m_0}(S_m)}{2^m}$ around $S_m$ and hence its limit, $S_m '$, is invertible. We now compute,
\ba \pi_{m+m_0+1,m+m_0}^{(n)}(S_{m+1}') & = & \lim_{k \to \infty} \pi_{m+m_0+1,m+m_0}^{(n)}(S_{m+1,k} ') \\
& = & \lim_{k \to \infty} \pi_{m+m_0+1,m+m_0}^{(n)}(\pi_{m+m_0+1+k}(S_{m+1+k})) \\
& = & \lim_{k \to \infty}S_{m,k+1}' = S_m '. \ea

This shows that the sequence, $S_m'$, is an inverse limit sequence and therefore defines an element, $S' \in \scr{O}_d (r) ^{n\times n}$. Moreover, since every $S_m '$ is invertible, $\pi_{m+m_0+1,m+m_0}^{(n)}(S_{m+m_0+1}^{';-1}) = S_m^{':-1}$ and we conclude that $S' \in \GL_n(\scr{O}_d (R))$. Lastly,
\[
F S_m ' = \lim_{k \to \infty} F S_{m,k} ' = \lim_{k \to \infty} ( \underbrace{\star}_{n - \ell} \  0_{\ell} )
\]
Similarly, the first $n-\ell$ coordinates of $S_m^{';-1} G$ are $0$. Therefore, $S'$ trivializes our relation.
\end{proof}

Observe that in the proof of Lemma \ref{lem:local_trivialization}, the existence of $G \in \scr{O} _d (r) ^n$ obeying $F \cdot G =0$ for $F \in \scr{O} _d (r) ^{1\times n}$ was only used to ensure that $F(qL)$ always has a non-trivial kernel for any $0<q<r$. In fact, we now have enough tools to show that $F \in \scr{O} _d (r) ^{1\times n}$ is injective if and only if it is injective on $q \cdot \rball$ for sufficiently large $0<q<r$.

\begin{lem} \label{lem:big_enough_injective}
Let $F \in \scr{O}_d ^{1\times n} (r)$, $r\in (0, +\infty]$. Then, the map $F \colon \scr{O}_d (r) \otimes \C^n \to \scr{O}_d (r)$ is injective if and only if $F(qL)$ is injective for sufficiently large $0<q<r$.
\end{lem}
\begin{proof}
Clearly, if $F$ is not injective, then $F(qL)$ is not injective for every $0<q<r$. Now assume that $F(qL)$ is not injective for every $0<q<r$. Let $\ell(q) = \nbdim \scr{W}_{\nbker F(qL)}$ be the wandering dimension of $\nbker F(qL)$. As in the proof of Lemma \ref{lem:local_trivialization}, the function $\ell(q)$ is monotonically decreasing, integer-valued, and bounded below by $1$ by our assumption. Therefore, there exists $0<r_0<r$, so that for $0< r_0 \leq q <r $, $\ell(q) =: \ell \in \N$ is constant. Therefore, for $0<r_0 \leq q < s < r$, the restriction map from $\nbker F(sL)$ to $\nbker F(qL)$ has dense range. In other other words, for every $\varepsilon > 0$ and every $G_q \in \nbker F(qL)$, there exists $H \in \nbker F(sL)$, such that $\|H(\tfrac{q}{s} L) 1 - G_q\| < \varepsilon$. Now we argue as in the proof of Theorem \ref{thm:entires_semifir}. Given the strictly increasing sequence $(t_m)_{m\in \N}$ so that $0<t_m <r$ and $t_m \uparrow R$, let $m_0 \in \N$ be the minimal value so that $t_{m_0} \geq r_0$. Then for any $m > m_0$, we inductively build sequences $\varepsilon_m > 0$ and $G_m \in \nbker F(mL)$ with the properties:
\bi
    \item[(i)] For every $m \in \N$, $\|G_{m+1}(\tfrac{t_m}{t_{m+1}} L) 1 - G_m\| < \varepsilon_m$;

    \item[(ii)] $\sum_{m=1}^{\infty} \varepsilon_m < +\infty$,

    \item[(iii)] For every $m \in \N$ and every $k \in \N$, $\varepsilon_{m+k} \leq \frac{\|G_m\|}{2^{m+k}}$.
\ei
Given such sequences, a Cauchy sequence argument, as in the proof of Theorem \ref{thm:entires_semifir} above, will produce an element of $\nbker F$. However, we must then also check that this element is not zero in order to show that $F$ is not injective. The third item above is designed to ensure that this limit is not $0$. To see this, fix $0 \neq G_1 \in \nbker F(t_{m_0 +1} L)$. Set $\varepsilon_1 = \min \{1, \tfrac{1}{2}\|G_1\| \}$. Assume that we have constructed $\varepsilon_1,\cdots,\varepsilon_m$ and $G_1,\cdots,G_m$ that satisfy the above three properties. Choose $G_{m+1} \in \nbker F(t_{m_0 + m + 1}L)$, such that $\|G_{m+1}(\tfrac{t_{m_0 + m}}{t_{m_0 + m+1}} L) 1 - G_m\| < \varepsilon_m$ and set $\varepsilon_{m+1} = \frac{1}{2^m} \min\{1, \tfrac{1}{2}\|G_1\|, \cdots, \tfrac{1}{2}\|G_m\| \}$. The properties (i--iii) that we demand are then satisfied just as in the proof of the previous Theorem \ref{thm:entires_semifir}. Now let $C_m = \lim_{k \to \infty} G_{m+k}(\tfrac{t_{m_0 + m}}{t_{m_0 + m+k}} L) 1$, which exists since the sequence on the right hand side is a Cauchy sequence. Moreover, by the third condition, $C_m$ is in the closed ball of radius $\tfrac{\|G_m\|}{2^m}$ around $G_m$ and, therefore, it is not zero. Since the inverse limit of the non-zero $C_m$ defines an element of $\scr{O}_d (r) \otimes \C^n$, this element must be non-zero. (Each $C_m$ is the restriction of this element to the NC row-ball of radius $0<t_{m_0+m}<r$.) Hence, $\nbker F$ is non-trivial and $F$ is not injective.
\end{proof}

The following theorem is a version of the Bergman Nullstellensatz for $\scr{O} _d (r)$, $r \in (0, +\infty]$.

\begin{thm} \label{thm:bergman}
Let $J \subseteq \scr{O}_d (r)$ be a finitely generated right ideal. Then, $J$ is closed and $J = I(\scr{Z}(J))$.
\end{thm}
\begin{proof}
If we show that $J$ is closed, then by Lemma \ref{lem:approx_nullstellensatz}, we can deduce that $J = I(\scr{Z}(J))$. Since $J$ is finitely-generated and $\scr{O} _d (R)$ is a semifir, $J$ is free as a right $\scr{O}_d (r)-$module. Namely, there exists $n \in \N$ and  $F \in \scr{O}_d (r) ^{1\times n}$, such that $F \cdot (\scr{O}_d (r) \otimes \C^n) = J$ and $F$ is injective. Assume that $F G_n \to h \in \scr{O}_d (r)$. Since $F$ is injective, by the preceding lemma, there exists $r_0 > 0$, such that for every $0<r_0 \leq q <r$, $F(qL)$ is injective. By increasing $r_0$, if necessary, we may apply Lemma \ref{lem:inner_restriction}. By the definition of the topology on $\scr{O}_d (r)$, we have that $F(qL) G_n(qL) \to h(qL)$ in $\A$ for any $0<q<r$. Let $0< r_0 \leq q <s <s' <r$. Write $F(s'L) = V T$, the inner--outer decomposition. Since $h \in \scr{O} _d (r)$, $h(s'L)$ belongs to $\A$ for any $0<s'<r$. In particular, since $V (L) T (L) G_n(s'L)1 \to h(s'L)1$ in $\hardy$,
$T (L)  G_n (s'L) \rightarrow V (L)^* h(s'L) 1$ in $\hardy$, so that $T(\frac{s}{s'} L) G_n (sL)$ converges in $\A$ for any $0< r_0 \leq q < s <s' <r$. By Lemma \ref{lem:inner_restriction}, $T(\tfrac{q}{s'}L)$ is square and invertible. Hence, $G_n(qL)$ converges in $\A$. In particular, for every $0< r_0 \leq q <r$, $G_n(qL)$ is bounded. Therefore, by Theorem \ref{thm:montel}, $G_n$ has a subsequence that converges in $\scr{O}_d (r) \otimes \C^n$ to some $G$. By the continuity of multiplication $h = F\cdot G \in J$. Hence, $J$ is closed.
\end{proof}

\begin{remark}
The stabilization of the wandering dimension used both in the proof of Lemma \ref{lem:inner_restriction} and Lemma \ref{lem:big_enough_injective} can be viewed as analogous to the algebraic Mittag--Leffler condition on inverse limits \cite[\href{https://stacks.math.columbia.edu/tag/0594}{Section 0594}]{stacks-project}.
\end{remark}

\section{The universal skew field of fractions of $\scr{O} _d (r)$} \label{sec:usfield}

By Theorem \ref{thm:entires_semifir}, each of the rings $\scr{O} _d (r)$, $r \in (0, +\infty]$ is a semifir, hence they each have a \emph{universal skew field of fractions} \cite[Theorem 7.5.14]{Cohn} which we denote by $\scr{M} _d (r)$. As before we write $\scr{M} _d$ in place of $\scr{M} _d (+\infty)$ for the universal skew field of fractions of the ring of uniformly entire NC functions, $\scr{O} _d$. 
In this section we describe these universal skew fields more concretely; namely we show that the skew fields $\scr{M} _d (r)$ consist of NC functions, which are in a natural sense, ``meromorphic" in $r \cdot \rball$.

The following description of universal skew fields of fractions is taken from \cite[Section 7.2]{Cohn}.
Given a ring $\scr{R}$, a field, $\scr{U}$, is said to be the universal skew field of fractions of $\scr{R}$ if (i) $\scr{R}$ embeds into $\scr{U}$ and $\scr{U}$ is generated as a skew field by the image of $\scr{R}$ under this embedding, and, (ii) $\scr{U}$ has the following universal property: If a matrix, $A \in \scr{R} ^{n \times n}$, becomes invertible under a homomorphism from $\scr{R}$, into a skew field, $\scr{D}$, then the image of $A$ in $\scr{U} ^{n\times n}$ is invertible. 
Equivalently, $\scr{U}$ can be characterized using the concept of a local homomorphism. 
If $\scr{E}, \scr{F}$ are skew fields, a \emph{local homomorphism} from $\scr{E}$ to $\scr{F}$ is a ring homomorphism $\varphi : \scr{R} _0 \subseteq \scr{E} \rightarrow \scr{F}$ whose domain, $\scr{R} _0 \subseteq \scr{E}$, is a local subring so that the kernel of $\varphi$ consists precisely of those elements in $\scr{R} _0$ that are not invertible in $\scr{R} _0$. 
A skew field, $\scr{U}$, is then the universal skew field of fractions of a ring, $\scr{R}$, if (i) $\scr{U}$ is generated by the embedding of $\scr{R}$ into $\scr{U}$, and (ii$'$) every homomorphism from $\scr{R}$ into a skew field, $\scr{D}$, extends to a local homomorphism from $\scr{U}$ into $\scr{D}$ whose domain contains $\scr{R}$. 
The universal skew field of fractions of a ring is unique up to isomorphism, if it exists \cite[Section 7.2]{Cohn}.

In the case where $\scr{R}$ is a semifir, there are additional characterizations of its universal skew field. First, given any ring, $\scr{R}$ and $A \in \scr{R} ^{m \times n}$, the \emph{inner rank} of $A$ is the smallest $k \in \N$ so that $A=BC$ with $B \in \scr{R} ^{m \times k}$ and $C \in \scr{R} ^{k \times n}$. A square matrix, $A \in \scr{R} ^{n \times n}$ is then said to be \emph{full}, if it has inner rank equal to $n$. If $\scr{R}$ is a semifir that embeds into and generates a skew field, $\scr{U}$, the following are equivalent:
\bi
    \item[(i)] $\scr{U}$ is the universal skew field of fractions of $\scr{R}$,
    \item[(ii)] the embedding $\scr{R} \hookrightarrow \scr{U}$ preserves the inner rank,
    \item[(iii)] every full, square matrix over $\scr{R}$ is invertible over $\scr{U}$,
\ei
see \cite[Theorem 7.5.13]{Cohn}. Moreover, \cite[Theorem 7.5.13]{Cohn} further asserts that if $\Sigma$ denotes the set of all square full matrices over $\scr{R}$, then $\scr{U} \simeq \scr{R} _\Sigma$ is the $\Sigma-$rational closure of $\scr{R}$ in $\scr{U}$. That is, elements of $\scr{U}$ can be identfied, up to a certain equivalence relation given by \cite[Section 7.4, Lemma 7.4.1]{Cohn}, with the entries of the inverses of full matrices over $\scr{R}$. It follows, by applying successive Schur complements (and potentially permutations to enable the Schur complements) to any full matrix, $A \in \scr{R} ^{n \times n}$, that any element of $\scr{U}$ can be identified with a certain equivalence class of NC rational expressions composed with elements of $\scr{R}$. 

Hence, if $f \in \scr{M} _d (r)$, we can write $f = \fr (h_1, \cdots, h_k ) \in \fskewk \circ \scr{O} _d (r) ^{1\times k}$ as an NC rational expression composed with elements of $\scr{O} _d (r)$. In this sense, elements of $\scr{M} _d (r)$ are ``meromorphic NC expressions" in $r \cdot \rball$, and $\scr{M} _d$ can be thought of as the skew field ``globally meromorphic NC expressions".  In particular, any $\fr \in \fskewk$ has a \emph{formal realization} or \emph{formal linear representation} of the form $(A,b,c)$, where $A=(A_0,\cdots, A_k)$ is a $k+1-$tuple of $m\times m$ complex matrices, and $b,c \in \C ^m$ for some $m \in \N$ \cite[Section 4.1]{freefield}\cite{CR-freefield}\cite[Section 4.1]{HMS-realize}. Namely, 
$$ \fr (\fz _1, \cdots, \fz _k ) =  b^* (A_0 + A_1 \fz _1 + \cdots + A_k \fz_k ) ^{-1}  c, $$ where the non-monic, affine \emph{linear pencil}, 
$$ A (\fz) := A_0 \cdot 1 + A_1 \fz _1 + \cdots + A_k \fz_k \in \C \langle \fz _1, \cdots, \fz _k \rangle ^{m \times m}, $$ is full over $\C \langle \fz_1, \cdots, \fz _k \rangle$, and hence invertible over $\fskewk$.

Since our semifirs, $\scr{O} _d (r)$, are unital, we have $\C ^m \subseteq \scr{O}_d (r) ^m$ as the constant vector-valued functions. Any element $f = \fr (h_1, \cdots, h_k) \in \scr{M} _d (r)$, $h_i \in \scr{O} _d (r)$, can then be represented as 
\[
    b^* A(\vec{h} (\fz) ) ^{-1} c; \quad \quad \vec{h} = (h_1, \cdots, h_k) \in \scr{O} _d (r) ^{1 \times k}, 
\]
for some formal realization, $(A,b,c)$, of $\fr \in \fskewk$. 

In addition to the semifirs, $\scr{O} _d (r)$, $r \in (0, \infty ]$, let $\scr{O} _d (0)$ denote the semifir of \emph{uniformly analytic NC germs at $0$} (initially defined and studied in \cite{KVV-local}), where $0 \in \C ^{1\times d} \subseteq \ncu$ denotes the origin of the NC universe. Namely, a uniformly analytic NC germ at $0$ is an equivalence class of uniformly analytic NC functions whose evaluations agree in some uniformly open neighbourhood of $0$. If $\fps = \C \langle \! \langle \fz _1, \cdots, \fz _d \rangle \! \rangle$ denotes the ring of all free formal power series, this is also a semifir by \cite[Proposition 2.9.19]{Cohn}, and we have the chain of embeddings of semifirs:
$$ \scr{O} _d (r) \hookrightarrow \scr{O} _d (q) \hookrightarrow \scr{O} _d (0) \hookrightarrow \fps, $$ for any $q < r \in (0, +\infty]$. An embedding of rings is said to be \emph{honest} if it preserves fullness of matrices. Each of the embeddings above can be viewed as an embedding into a skew field, \emph{e.g.} the embedding $\scr{O} _d (r) \hookrightarrow \scr{O} _d (q)$ can be viewed as an embedding $\scr{O} _d (r) \hookrightarrow \scr{M} _d (q) \supsetneqq \scr{O} _d (q)$. Hence, if any of these embeddings are honest then they preserve the inner rank of all matrices by \cite[Corollary 5.4.10]{Cohn}. By \cite[Proposition 5.3]{KVV-local}, the final embedding, $\scr{O} _d (0) \hookrightarrow \fps$ is totally inert, hence honest. (The concept, \emph{totally inert} is a stronger condition that implies honesty. We will prove below that each embedding $\scr{O} _d (r) \hookrightarrow \scr{O} _d (q)$ for any $0 \leq q < r \leq +\infty$ is \emph{semi-inert} and hence honest.) We remark that since $\scr{O} _d (r) \hookrightarrow \scr{O} _d (0)$, the evaluations of elements of $\scr{O} _d (r)$ in $r \cdot \rball$ are generically invertible by \cite[Theorem 5.7]{KVV-local}. 

Since elements of $\scr{M} _d (r)$, $r \in (0, +\infty]$ are non-commtuative rational expressions in $\scr{O} _d (r)$, they can reasonably be called ``NC meromorphic expressions". However, there could exist, in principle, a non-zero element $0\neq f \in \scr{M} _d (r)$, with $0 \in \nbdom f$, which evaluates identically to $0$, everywhere it is defined in $r \cdot \rball$. That is, such an $f$ would be a ``meromorphic identity" in the sense of Amitsur \cite{Amitsur}. Here, a celebrated theorem of Amitsur implies that there are no ``rational identities": If $\fr \in \fskew$ is an NC rational function that evaluates to $0$ everywhere it is defined in $\ncu$, then $\fr =0$ \cite[Theorem 14]{Amitsur} \cite[Chapter 8]{Rowen80}. If $0 \neq f \in \scr{M} _d (r)$ is a non-trivial NC meromorphic identity, then $0 \neq f^{-1} \in \scr{M} _d (r)$ has empty domain in the NC universe and hence neither $f$ nor $f^{-1}$ can be interpreted as analytic non-commutative functions in the sense of NC function theory. Fortunately, we are able to extend Amitsur's theorem to $\scr{M} _d (r)$, $r \in (0, +\infty]$ to prove that there are no non-trivial NC meromorphic identities, and hence elements of $\scr{M} _d (r)$ can be identified with bona fide analytic and holomorphic non-commutative functions whose domains are analytic-Zariski dense at all sufficiently high levels of $r \cdot \rball$, see Corollary \ref{cor:Amitsur Thm for Mdr}. In order to establish our Amitsur Theorem for $\scr{M}_d(r)$, we first prove a generalization of Cohn's Inertia Theorem \cite[Theorem~2.9.16]{Cohn}.

\subsection{Semi-Inertia Theorem}

Let $\nu$ be the standard order function on $\fps$: namely, if $f(\fz) = \sum _{\om \in \word} \hat{f} _\om \fz ^\om$, then $\nu(f) = \inf\{\abs{\om}  : \, \hat{f}_\om \neq 0\}$ and $\nu(0) = +\infty$.
Moreover, for any $f = (f^1,\dots, f^n)\in \fps^n$, we define $\nu(f) = \min_{1\leq i\leq n}\nu(f^i)$.

Extending our idea from the Hardy space, recall that for any word $w\in \langle\fz\rangle$, $L^w$ is the operator of left multiplication by $w$, while $L^{w*}$ denotes the backward left shift by $w$: for any word $v$, if $v=  wu$, then $L^{w*}v = u$, while $L^{w*}v = 0$ if $v$ cannot be written as $wu$ (the terminology employed by Cohn is to call $L^{w*}$ a \emph{left transduction}, however we will remain consistent with out previous terminology.)
Notably, if $\abs{w}=\abs{u}$, then $L^{w*}L^u = \delta_{u,w}$.
If $f\in \fps^{n}$ has $\nu(f) = \min_i \nu(f^i) \geq k >0$, then 
\[
    f = \sum_{\abs{u} = k} uf_u = \sum_{\abs{u} = k} L^uf_u
\]
where each $f_u\in \fps^{n}$.
Consequently, for any $g\in \fps^n$ and word $\abs{w} = k$,
\begin{equation}
    \label{eq:left transduction of product}
    L^{w*}(fg) = \sum_{\abs{u}=k} L^{w*}L^u f_u g = f_wg = L^{w*}(f)g.
\end{equation}

\begin{lemma}
    \label{lem:add transductions}
    Let $S = \fps$ and $\fp\subset R\subset S$, with $R$ closed under left transductions.
    Suppose $A\subset S^{1\times n}$ is a finitely-generated left $R$-module and $B\subset S^n$ is a finitely-generated right $R$ module.
    If for each $a\in A$ and $b\in B$ we have $ab\in R$, then there exists a finitely-generated left $R$-module $\hat{A}$ of $S^{1\times n}$ so that
    \begin{enumerate}
        \item $A\subset \hat{A}$;
        \item for each $a\in \hat{A}$ and $b\in B$, we have $ab\in R$;
        \item if $\hat{A}_0 = \{a(0) \, : \, a\in \hat{A}\}$, then either $\hat{A}_0 = \C^n$, or there exists $T\in \GL_n(\C)$ with $\hat{A}T = \tilde{A}\oplus 0$.
    \end{enumerate}
\end{lemma}

\begin{proof}
    If $A = 0$ or if $A_0 = \{a(0) \, : \, a\in A\} = \C^n$ then we are done, so let us assume that $A\neq 0$ and $A_0\cong \C^s$ where $s<n$.
    If $a\in A$ and $b\in B$, then $ab\in R$, and since $R$ is closed under left transductions, $L^{w*}(ab)\in R$ as well.
    However, if $w\in\word$ with $\abs{w}=\nu(a)$, then Equation~\eqref{eq:left transduction of product} shows that $L^{w*}(a)b =L^{w*}(ab)\in R$.
    
    Take any $\hat{a}\in A$ with order $k>0$ and note that $L^{w*}(\hat{a})b\in R$ for all $b\in B$, as noted above.
    Hence, by adjoining all the left transductions of $\hat{a}$ of order $k$ and taking the left $R$-linear span, we obtain a finitely-generated left $R$-module that has at least one element of order $0$, and whose products against $B$ end up in $R$.
    This is our starting $\hat{A}$ (we may iteratively apply this type of argument to enlarge it again).
    
    Now, $\hat{A}_0 = \{a(0) \, : \, a\in \hat{A}\} \cong \C^s$ with $1\leq s < n$ and take a basis $v_1,\dots, v_s\in \C^n$.
    There must exist elements $\alpha_1,\dots, \alpha_s\in \hat{A}$ so that $\alpha_i(0) = v_i$.
    Take any element $a\in \hat{A}$ and note that $a(0)\in \hat{A}_0$.
    Assume there exists a minimal $k\geq 1$ and word $w\in \word$ with $\abs{w} = k$ so that $(L^{w*}a)(0)\notin \hat{A}_0$.
    It follows that if $\abs{u}<k$, then $(L^{u*}a)(0) \in \hat{A}_0$, hence we can find polynomials $p_1,\dots, p_s$ with degree less than $k$ so that $\nu(a - p_1\alpha_1 - \dots - p_s\alpha_s) = k$.
    Naturally, $a - p_1\alpha_1 - \dots - p_s\alpha_s\in \hat{A}$ and by assumption, $(L^{w*}a)(0)\notin \hat{A}_0$.
    Thus, as discussed before, we may enlarge $\hat{A}$ by adjoining this left transduction, yielding a new $\hat{A}$ whose constant coefficients span a space with dimension at least $s+1$.
    
    On the other hand, it could be the case that whenever $a\in \hat{A}$, and $w\in \word$, then $(L^{w*}a)(0)\in \hat{A}_0$.
    Take any matrix $T\in \GL_n(\C)$ so that $v_iT = e_i$, the $i^{\text{th}}$ standard basis vector in $\C^{1\times n}$.
    Hence, $(L^{w*} a)(0)T\in \C^{1\times s}\oplus 0_{n-s}$ so that $aT = (a^1,\dots, a^s, 0,\dots, 0)$.
    Thus, $\hat{A}_0T = \C^s\oplus 0_{n-s}$ and $\hat{A}T = \tilde{A}\oplus 0_{n-s}$, for some finitely-generated left $R$-module $\tilde{A}$.

    By iterating this process a finite number of times we always arrive at one of our two desired cases; either we have a finitely-generated left $R$-module $\hat{A}$ so that $A\subset \hat{A}\subset S^{1\times n}$ with $\hat{A}_0 = \C^n$, or there exists $T\in \GL_n(\C)$ so that $\hat{A}T = \tilde{A}\oplus 0_{n-s}$ and $\hat{A}_0T = \C^{1\times s}\oplus 0_{n-s}$.
\end{proof}

\begin{defn}
    Suppose $R\subset S$ are rings.
    We say $a\in S^{1\times n}$ and $b\in S^n$ \textbf{lie trivially} in $R$, if for each $1\leq i\leq n$, either $a_i$ and $b_i$ are in $R$, or $a_i = 0$ or $b_i = 0$.
    Further, the subring $R$ of $S$ is said to be \textbf{n-semi-inert} in $S$ if for all $m\leq n$ and any finite families $(a_i)_{i=1}^k$ of rows in $S^{1\times m}$ and of columns $(b_j)_{j=1}^{\ell}$ in $S^m$ such that $a_ib_j\in R$ for all $1\leq i\leq k$ and $1\leq j\leq \ell$, there exists $U\in \GL_m(S)$ such that on writing $a_i' = a_iU$, $b_j'= U^{-1}b_j$, each product $a_i'b_j'$ lies trivially in $R$.
    If $R$ is $n$-semi-inert in $S$ for all $n\in \N$, then we say $R$ is \textbf{semi-inert} in $S$.
\end{defn}

Our definition of semi-inert is weaker than the notion of totally inert, which allows for the families $(a_i)$ and $(b_i)$ to be arbitrary.
However, the following Lemma shows that semi-inertness is sufficient to preserve inner ranks.

\begin{lemma}
    \label{lem:semi-inert implies honest}
    Suppose $R\subset S$ are rings.
    If $R$ is semi-inert in $S$, then the embedding $R\hookrightarrow S$ preserves inner ranks.
    In particular, a matrix over $R$ is full if and only if it is full over $S$.
\end{lemma}

\begin{proof}
    Take a matrix $A\in R^{m\times n}$ and assume it has a rank factorization over $S$: $A = BC$, where $B\in S^{m\times r}$ and $C\in S^{r\times n}$.
    By assumption $R$ is semi-inert in $S$, hence there exists an invertible matrix $U\in S^{r\times r}$ so that $(BU)$ and $(U^{-1}C)$ lie trivially in $R$.
    If the $i^{\text{th}}$ entry of any row of $BU$ is not in $R$, then it follows that $U^{-1}C$ has a column of zeros.
    However, this would allow a smaller factorization of $A$, which contradicts $BC$ being a rank factorization for $A$ over $S$.

    Thus, $(BU)\in R^{m\times r}$ and $(U^{-1}C)\in R^{r\times n}$, so it is a rank factorization for $A$ over $R$.
    Therefore, the embedding of $R$ into $S$ preserves the inner rank.
\end{proof}

We now establish our ``Semi-Inertia Theorem" on transduction invartiant sub-semifirs of $\fps$ as a generalization of Cohn's ``Intertia Theorem" on $\fp\subset \fps$ \cite[Theorem 2.9.16]{Cohn}.

\begin{thm}[Semi-Inertia Theorem]
    \label{thm:semi-inertia theorem}
    Let $S = \fps$ and $\fp\subset R\subset S$ be a semifir with maximal ideal $(x)$.
    If $R$ is closed under left and right transductions, then $R$ is semi-inert in $S$.
\end{thm}

\begin{proof}
    Suppose $(f_i)\subset S^{1\times n}$ and $(g_j)\subset S^n$ are finite sets such that $f_ig_j\in R$ for each $i,j$.
    Let $A$ be the finitely-generated left $R$ module of $S^{1\times n}$ generated by $(f_i)$, and let $B$ be the finitely-generated right $R$ module of $S^{n}$ generated by $(g_i)$.
    If $A = \tilde{A}\oplus 0_{n-s}$, then for any $a\in A$ and $b\in B$, whether $a$ and $b$ lie trivially in $R$ depends entirely on the first $s$ coordinates of each, so we simply reframe the problem for the smaller size.

    Apply Lemma~\ref{lem:add transductions} to enlarge $A$, and as remarked above, we may assume $A_0 = \{a(0) \, : \, a\in A\} = \C^n$.
    In particular, there are elements $\alpha_1,\dots, \alpha_n\in A$ so that $\alpha_i(0) = e_i$.
    Note $\alpha_i^i(0)\neq 0$, and since $S$ is a local ring, $\alpha_i^i(0)$ is invertible over $S$.
    Thus, by performing column operations on $A$ (and the corresponding row operations on $B$), we may assume that $e_1,\dots, e_n\in A$.
    In other words, $R^{1\times n}\subset A\subset S^{1\times n}$.
    Since $AB \subset R$, and $e_ib = b^i$, it follows that $B\subset R^n$.

    We may assume  $\{b(0) \, : \, b\in B\} = \C^n$ by applying the right-handed version of Lemma~\ref{lem:add transductions} to enlarge $B$ while keeping it a finitely-generated right $R$ module of $R^n$ (the row operations in Lemma~\ref{lem:add transductions} are over $\C$, so the enlarged version of $B$ is still a subset of $R^n$).

    However, $R$ is a \textit{semifir} and $B$ is a finitely-generated right $R$-submodule of the free right $R$-module $R^n$, hence by \cite[Theorem 2.3.11]{Cohn}, it follows that $B$ is also a free right $R$ module, that is, $B = R^m$ for some $m$.
    Let $\bar{B}$ be the image of $B$ in $(R/(x))^n$.
    We have the following chain of equalities and module isomorphisms:
    \begin{align*}
        (R/(x))^n &= \bar{B} = B/B(x) \cong B\otimes_R R/(x) = R^m\otimes_R R/(x) \cong (R\otimes_R R/(x))^m \\ 
            &= (R/(x))^m.
    \end{align*}
    Since $R/(x)\cong \C$ (we had assumed $(x)$ is a maximal ideal), it follows that $m = n$ and $B = R^n$.
    
    Now, take $\{b_1,\dots,b_n\}\subset \R^n$, a right $R$ basis for $B$, and let $U\in M_n(R)$ be the matrix whose columns are the basis elements.
    Since the image of $B$ in $R/(x)$ is $\C^n (\cong (R/(x))^n)$, it follows that the columns of $U(0)$ form a $\C$-basis, so $U(0)$ is invertible over $\C$.
    Once more, $S$ is a local ring, so $U$ is invertible over $S$, that is, $U^{-1}\in M_n(S)$.
    Moreover, $B = UR^n$, hence $U^{-1}B = R^n$.
    Finally, since $AB\subset R$ and $e_1,\dots, e_n\in U^{-1}B = R^n$, we see that $aUe_i = (aU)^i\in R$ for all $a\in A$ and $1\leq i\leq n$.
    Therefore $AU = R^{1\times n}$ and $U^{-1}B = R^n$, as desired.
\end{proof}

Recall that $\scr{O} _d (0)$ denotes the semifir of uniformly analytic NC germs at $0$, as defined and constructed in \cite{KVV-local}.

\begin{thm}
    \label{thm:NC Entire Amitsur}
    For $0\leq q<r\leq +\infty$, each embedding $\scr{O}_d(r)\hookrightarrow \scr{O}_d(q ) \hookrightarrow \fps$ is honest.
\end{thm}

By \cite[Proposition 5.3]{KVV-local}, the embedding $\scr{O} _d (0) \hookrightarrow \fps$ is in fact, totally inert, hence honest.

\begin{proof}
    For $r\in [0,\infty]$, $\fp \subset \scr{O}_d(r) \subset \fps$ is a semifir that is closed under left/right transductions, hence Theorem~\ref{thm:semi-inertia theorem} implies $\scr{O}_d(r)$ is semi-inert in $\fps$.
    Hence, Lemma~\ref{lem:semi-inert implies honest} implies $\scr{O}_d(r) \hookrightarrow \fps$ preserves inner ranks.
    Thus, the inner rank of a matrix is the same for all our rings, and consequently, each of the written embeddings is honest.
\end{proof}

It follows from \cite[Proposition 2.1]{Cohn-freeprod} and \cite[p.285]{Cohn-freerings} that any honest embedding of semifirs, $\scr{R} \hookrightarrow \scr{R} '$, extends to a homomorphism of their universal skew fields of fractions. Any homomorphism of skew fields is necessarily injective, so that one can identify $\scr{U}$, the universal skew field of $\scr{R}$, with a skew sub-field of $\scr{U} '$, the universal skew field of $\scr{R} '$. Hence, we obtain the following:
\begin{cor}
For any $0 < q < r < + \infty$, we have the following chain of proper embeddings:
$$ \scr{M} _d \hookrightarrow \scr{M} _d (r) \hookrightarrow \scr{M} _d (q) \hookrightarrow \scr{M} _d (0) \hookrightarrow \skewfps. $$ 
\end{cor}

\subsection{Amitsur's Theorem for $\scr{M}_d(r)$}

Before moving on to the proof of the meromorphic version of Amitsur's Theorem, we establish a few facts about how elements of $\scr{M}_d(r)$ evaluate on generic matrices.

We let $\{\xi_{ij}^{(k)}\}_{i,j,k}$ be a collection of commuting indeterminants, and for each $n\in \N$ and $1\leq k\leq d$, we take $\Xi^{(k)}_n$ to be the $n\times n$ matrix whose $i,j^{\text{th}}$ entry is $\xi_{ij}^{(k)}$.
Fix $d\geq 1$.
The matrices $\Xi^{(1)}_n, \dots, \Xi^{(d)}_n$ are ($n\times n$) generic matrices.
For any $n$ we set $\underline{\xi}^{n} = \{\xi^{(k)}_{ij}\}_{i,j=1,k=1}^{n,d}$ and $\gentup{n} = (\Xi_1^{(n)},\dots, \Xi_d^{(n)})$.
Moreover, $\GM_n$, the $n\times n$ generic matrix algebra, is the unital $\C$-algebra generated by $\gentup{n}$.
Naturally, $\GM_n$ is a subalgebra of $(\C[\genvars{n}])^{n\times n}$. Let us write $\canr{n}$ for the algebra of analytic functions on $r \bdn$ equipped with the usual topology of uniform convergence on compact subsets (naturally, we use $\genvars{n}$ for the variables of $\canr{n}$, as well). 
We will denote by $\cmeror{n}$ the field of fractions of $\canr{n}$, which is the field of meromorphic functions on $r \bdn$.
Let $\overline{\GM_n}(r)$ be the completion of $\GM_n$ in the algebra $(\canr{n})^{n \times n}$. The following two lemmas are the analytic versions of \cite[Lemma 2.2]{HKV-factorization} and \cite[Proposition 3.4]{KVV-local}.

\begin{lem} \label{lem:ud_n-skew_field}
The ring $\overline{\GM_n}(r)$ is a prime ring. In particular, the ring $\overline{\UD_n}(r)$ of central quotients of $\overline{\GM_n}(r)$ is a skew field. Moreover, $\overline{\GM_n}(r)$ agrees with the subring of $\GL_n(\C)$ concommitants in $\canr{n}^{n \times n}$
\end{lem}
\begin{proof}
    Consider $\overline{\GM_n}(r) \subset \cmeror{n}^{n\times n}$.
    We note that if $X \in r \bdn$ is an irredcubile point, then the evaluation map at $X$ is a surjective homomorphism $\C \langle \fz \rangle \to \C^{n \times n}$. 
    Hence, we may choose $f_{ij} \in \C \langle \fz \rangle$ such that $f_{ij}(X) = E_{ij}$.
    Now take ideals $A,B\subset \overline{\GM_n}(r)$ such that $AB = \{0\}$, and take any elements $a\in A$ and $b\in B$.
    Note that $f_{1i}af_{j1}\in A$ and $f_{1k}af_{\ell 1}\in B$, and their evaluations at $X$ yield $a_{i,j}(X)E_{11}$ and $b_{k,\ell}(X)E_{11}$.
    Since the product of these elements must be $0$, it follows that $a_{i,j}(X)b_{k,\ell}(X) = 0$ for all $i,j,k,\ell$.
    However, this can be done for any irreducible $X\in r \bdn$ and since the irreducible $n\times n$ tuples are Euclidean open, it follows that $\alpha_{i,j}b_{k,\ell} = 0$.
    Thus, either $a$ or $b$ was the zero matrix, and it follows that either $A = \{0\}$ or $B = \{0\}$.
    Hence $\overline{\GM_n}(r)$ is a prime ring.
    
    Hence, the center of $\overline{\GM_n}(r)$ is a domain and by Posner's theorem \cite[Theorem 1.7.9]{Rowen80}, $\overline{\UD_n}(r)$ is a central simple algebra of degree $n$. 
    Using the Taylor expansion argument in \cite[Proposition 3.4]{KVV-local} implies that $\overline{\UD_n}(r)$ is a skew field.

    To prove that $\overline{\GM_n}(r)$ is the subring of $\GL_n(\C)$ concommitants in $\canr{r}^{n \times n}$, we assume that $F \in \canr{r}^{n \times n}$ is a $\GL_n(\C)$ concommitant. 
    Write $F = \sum_{m=0}^{\infty} F_m$, where $F_m$ is the $m$-th homogenous component of the Taylor expansion of $F$ at the origin. 
    Since the expansion converges in the topology of $\canr{r}^{n \times n}$, it suffices to prove that each $F_m \in \GM_n$. 
    The latter claim is immediate by writing $F_m(\genvars{n}) = \frac{1}{2\pi} \int_0^{2 \pi} e^{- i m \theta} F(e^{i \theta} \underline{\xi}^n) d\theta$ and using the concomitant relation.
\end{proof}

\begin{lem} \label{lem:full_inv_dom}
   If $A \in \cO_d(r)^{\ell \times \ell}$, then the following are equivalent:
   \begin{itemize}
       \item[(i)] $A$ is full;

       \item[(ii)] There exists $n \in \N$ and $X \in r \bdn$, such that $\det(A(X)) \neq 0$;

       \item[(iii)] There exists $n_0 \in \N$, such that for all $n \geq n_0$, $\det A(\Xi^{n}) \neq 0$.
   \end{itemize}
\end{lem}
\begin{proof}
    The implications (iii) $\Rightarrow$ (ii) $\Rightarrow$ (i) are clear. The implication (i) $\Rightarrow$ (ii) follows from Theorem \ref{thm:NC Entire Amitsur}. 
    Lastly, to prove (ii) $\Rightarrow$ (iii) we assume that there exists some $n_0 \in \N$ and $X \in r \mathbb{B}^{(n_0\times n_0)\cdot d}$, such that $\det A(X) \neq 0$. 
    By replacing $X$ with $X^{\oplus k}$ if necessary, we may assume that $n_0 \geq \ell$. 
    It suffices to prove that $A(\Xi^{n_0 + 1}) \neq 0$ and then proceed inductively. 
    Let $\hat{X} = X \oplus 0$ and note that $\rank A(\hat{X}) \geq n_0 \ell$. 
    For the sake of brevity, we let $\mathbb{U} = \overline{\UD_{n_0+1}}(r)$ and $\mathbb{L} = \cmeror{n_0+1}$.
    Now we view 
    \[
        A(\Xi^{n_0 + 1}) \in  \mathbb{U} ^{\ell \times \ell} 
        \subseteq \mathbb{L}^{(n_0 + 1) \ell \times (n_0 + 1) \ell}.
    \]
    Considering $n_0 \ell \times n_0 \ell$ minors, we see that $n_0 \ell \leq \rank_{\mathbb{L}} A(\Xi^{n_0 + 1})$.

    By the preceding lemma, $\mathbb{U}$ is a skew field. 
    Hence, as in \cite[Proof of Theorem 1.7]{DerkMak-rank}, we use row and column operations over $\mathbb{U}$ to find $P,Q\in \GL_{\ell}(\mathbb{U})$, such that $P A(\gentup{n_0+1})Q = I^{\mathbb{U}}_\rho \oplus 0$, where $\rho = \rank_{\mathbb{U}}(A(\gentup{n_0+1})) \leq \ell$, and $I^{\mathbb{U}}_\rho$ is the $\rho \times \rho$ identity matrix (over $\mathbb{U}$).
    Since the multiplicative identity of $\mathbb{U}$ is the standard identity matrix $I_{n_0+1}\in \C^{(n_0+1)\times (n_0+1)}$, we see treating $I^{\mathbb{U}}_{\rho}$ as a matrix over $\mathbb{L}$ gives us $I_{(n_0+1)\rho}$.
    Moreover, $P,Q$ are invertible as $(n_0+1)\ell \times (n_0+1)\ell$ matrices over $\mathbb{L}$, so 
    \begin{align*}
        \rank_{\mathbb{L}}\left( A(\gentup{n_0+1}) \right) 
        &= \rank_{\mathbb{L}} \left( PA(\gentup{n_0+1})Q \right) 
        = \rank_\mathbb{L}\left( I^{\mathbb{U}}_\rho \oplus 0 \right)
        = \rank_{\mathbb{L}}(I_{(n_0+1)\rho} \oplus 0) \\
        &= (n_0+1)\rho.
    \end{align*}
    We summarize this as
    \[
    n_0 \ell \leq (n_0 + 1) \rho \leq (n_0 + 1) \ell.
    \]
    If $\rho \leq \ell - 1$, then $n_0 \ell \leq (n_0 + 1)(\ell - 1)$. 
    Hence, $n_0 \leq \ell-1$, which contradicts our assumption that $n_0 \geq \ell$.
    Thus, $\rank_{\mathbb{L}}(A(\gentup{n_0+1})) = (n_0+1)\ell$ and $A(\gentup{n_0+1})$ is invertible.
    Therefore, the proof follows by induction.
\end{proof}

\begin{cor}[Amitsur's Theorem for $\scr{M}_d(r)$]
    \label{cor:Amitsur Thm for Mdr}
    There are no non-zero meromorphic identities for $\scr{M}_d(r)$, $r\in [0,+\infty]$.
    That is, if $f \in \scr{M} _d (r)$, its domain is a uniformly open NC subset of $r \cdot \rball$ that is analytic-Zariski dense at all sufficiently high levels, and if $f$ evaluates to $0$ identically on its domain, then $f=0 \in \scr{M} _d (r)$. 
\end{cor}

\begin{proof}
By \cite[Section 7.4]{Cohn}, any $f \in \scr{M} _d (r)$ can be represented by a quadruple, $f \sim (A,b,c,d)$, 
$$ f = d + bA^{-1} c, $$ where $d \in \scr{O} _d (r)$, $b \in \scr{O} _d (r) ^{1\times \ell}$, $c \in \scr{O} _d (r) ^\ell$ and $A \in \scr{O} _d (r) ^{\ell \times \ell}$ is full. Since each embedding, $\mr{e} _r : \scr{O} _d (r) \hookrightarrow \scr{O} _d (0)$ is honest, the image of any full matrix, $A \in \scr{O} _d (r) ^{\ell \times \ell}$ is full as a matrix over $\scr{O} _d (0)$. By \cite[Theorem 5.7]{KVV-local}, since $\mr{e}_r (A) \in \scr{O} _d (0) ^{\ell \times \ell}$ is full, given any uniformly open neighbourhood, $\Om$, of $0$, there exists a $Y \in \Om _n$ so that $A(Y)$ is invertible. 

The matrix $A\in \scr{O} _d (r) ^{\ell \times \ell}$ is analytic on $r \cdot \bdn$, and since the determinant is a polynomial in the entries of a matrix, $\det(A):r \cdot \bdn\to \C$ is analytic.
Hence, $\det A$ is nonzero on an open subset of $r \cdot \bdn$ and if $\det A$ vanishes on a nonempty and open subset of $r \cdot \bdn$, then $\det A \equiv 0$ on $r \cdot \bdn$. That is, $A$ is invertible on an (analytic-Zariski and matrix-norm) open and dense subset of $r \cdot \bdn$ for all sufficiently large $n \in \N$ by Lemma \ref{lem:full_inv_dom}.

Since $A : r \cdot \rball \rightarrow \left( \C ^{\N \times \N} \right) ^{\ell \times \ell}$ is uniformly analytic, it is a continuous function between these two topological spaces, where $\left( \C ^{\N \times \N} \right) ^{\ell \times \ell}$ is also equipped with the uniform topology. Since the set of all invertible $n \times n$ matrices is open in the matrix-norm topology, it is readily seen that the set of all invertible, square matrices in $\left( \C ^{\N \times \N} \right) ^{\ell \times \ell}$ is a uniformly open NC set, $\mr{GL} _{\N \cdot \ell} = \bigsqcup _{n=1} ^\infty \mr{GL} _{n \ell}$, and hence the pre-image of this set under $A$ in $r \cdot \rball$ is also open in the uniform topology, by continuity of $A$.
Thus, $A$ is invertible on a uniformly open subset of $r \cdot \rball$ that is analytic-Zariski dense at all sufficiently high levels.
On the other hand, it is also readily checked that $f=d +b A^{-1} c$ is locally bounded in the uniform topology on this uniformly open NC set,
\[
\nbdom f = \left\{ \left. X \in r \cdot \rball \right| \ A(X) ^{-1} \, \exists \right\},
\]
hence uniformly analytic.

 
It follows that $\scr{M} _d (r)$ contains no meromorphic identities: Indeed, suppose that $0 \neq f \in \scr{M} _d (r)$ is a meromorphic identity. By the above arguments, $\nbdom f \subseteq r \cdot \rball$ is analytic-Zariski dense at all sufficiently high levels and uniformly open, hence $f$ vanishes identically on this domain. However, then $0 \neq f^{-1} \in \scr{M} _d (r)$ cannot be defined on a uniformly open NC domain that is analytic-Zariski dense at all sufficiently high levels, contradicting the above arguments. 
\end{proof}

\begin{remark}
    \label{rem:evaluation equiv classes}
    If $f\sim (A,b,c,d)$ and $f'\sim (A',b',c',d')$ agree on their common domain, then $f-f'$ evaluates to zero on a uniformly open set, hence Corollary~\ref{cor:Amitsur Thm for Mdr} implies $f = f'$.
    This means that rather than defining $f\in \scr{M}_d(r)$ as an equivalence class of representations $(A,b,c,d)$, we may define $f$ as an evaluation equivalence class of rational expressions composed with elements of $\scr{O}_d(r)$ such that the composition is defined for at least one point in $r\cdot \rball$. That is, $f$ can be defined as the equivalence class of all such expressions whose evaluations agree on the (necessarily non-empty) intersection of their domains in $r \cdot \rball$.
\end{remark}

\section{Evaluation in stably finite algebras} \label{sec:application}

Recall that $\scr{M} _d$ denotes the universal skew field of the semifir of uniformly entire NC functions, $\scr{O} _d$. (In this section, although we focus on $\scr{M} _d = \scr{M} _d (+\infty)$; our results here extend readily to $\scr{M} _d (r)$ for any $r>0$.) By a theorem of P.M. Cohn, elements of the free skew field, $\fskew$, \emph{i.e.,} NC rational functions, have well-defined evaluations and domains in any stably finite (unital) algebra over $\C$ \cite[Theorem 7.3.2]{Cohn2}, \cite{Cohn-uni}. (Also see \cite[Theorem 4.4]{freefield}.) Our goal in this section is to extend this result to evaluation of uniformly meromorphic NC functions in a suitable class of stably-finite algebras.

As discussed in the opening discussion of Section~\ref{sec:usfield}, any element in $\scr{M} _d$ is an equivalence class of NC rational expressions composed with elements of $\scr{O} _d$. Namely, if $f = \fr \circ (h_1, \cdots, h_k)$, where $\fr \in \fskewk$ and $h_i \in \scr{O} _d$, then there exist $b,c \in \C ^n$ and $A \in \C ^{n\times n} \otimes \C ^{1 \times (k+1)}$ for some finite $n \in \N$ so that,
\be f (\fz_1, \cdots, \fz_d) = b^* A (h_1 (\fz), \cdots, h_k (\fz ) ) ^{-1} c, \label{Mdexpr}, \ee where 
$$ A(\fz_1, \cdots, \fz _k) := 1\otimes A_0 +\fz _1 \otimes A_1 + \cdots + \fz _k \otimes A_k, $$ is the affine linear pencil of $A$. Setting $\vec{h} := (h_1, \cdots, h_k)$, we will write
$A (h_1 (\fz), \cdots, h_k (\fz ) ) = A (\vec{h} (\fz))$. Any element, $f \in \scr{M} _d$ is then an equivalence class of such rational expressions in $\scr{O} _d$, where the equivalence relation is described in \cite[Lemma 7.4.1]{Cohn}. If $f \in \scr{M} _d$ is such that there exists a $k \in \N$ and $(A,b,c) \in \C ^{(n\times n) \cdot (k+1)} \times \C ^n \times \C ^n$ and $\vec{h} \in \scr{O} _d ^{1\times k}$ so that the expression given in Equation (\ref{Mdexpr}) is in the equivalence class of $f$, we will write $f \sim (A,b,c; \vec{h})$ and we will say that $(A,b,c; \vec{h})$ is a \emph{formal representation} of $f$.   

Let $\scr{A}$ be a (unital) stably--finite algebra over $\C$. We say that a $d-$tuple of elements of $\scr{A}$, $a := (a_1, \cdots, a_d)$ belongs to the $\scr{A}-$domain of $f$, $\mr{Dom} _\scr{A} (f)$, if there exists: (i) a unital homomorphism $\varphi _a : \scr{O} _d \rightarrow \scr{A}$ obeying $\varphi _a (\fz _i) = a_i$ for each $1\leq i \leq d$, and (ii), 
a formal representation $f \sim (A,b,c; \vec{h})$ so that the linear pencil $A \circ \vec{h} (a) := A \circ \varphi _a (\vec{h})$ is invertible in $\scr{A} ^{n\times n} =\scr{A} \otimes \C ^{n\times n}$. One then defines 
$$ f (a) := 1_\scr{A} \otimes b^* A (\vec{h} (a) ) ^{-1} 1 _\scr{A} \otimes c. $$ Similarly, given a formal representation, $(A,b,c; \vec{h}) \in \C ^{(n\times n)\cdot (k+1)} \times \C ^n \times \C ^n \times \scr{O} _d ^k$, of an element of $\scr{M} _d$, we will say that the $d-$tuple, $a=(a_1, \cdots, a_d) \in \scr{A} ^{1\times d}$ belongs to $\scr{D} _{\scr{A}} (A; \vec{h})$, if there is a homomorphism, $\varphi _a$, obeying $\varphi _a (\fz _i) = a_i$, as before, and if the affine linear pencil $A(\vec{h} (a))$ is invertible in $\scr{A} ^{n \times n}$.

Of course, there is the question of whether the value of $f (a)$ is well defined. Namely, it is not obvious whether or not this definition of $f(a)$, for $a \in \mr{Dom} _\scr{A} (f)$, is independent of the choice of formal representation for $f$. That this is indeed always well-defined is the essence of the following theorem.

\begin{thm} \label{thm:stabeval}
Let $\scr{A}$ be a stably--finite algebra over $\C$. Given $f \in \scr{M} _d$ with formal representations $(A,b,c; \vec{h})$ and $(A',b',c'; \vec{h}')$, suppose that $a = (a_1, \cdots, a_d) \in \scr{A} ^{1 \times d}$ belongs to $\scr{D} _\scr{A} (A ; \vec{h}) \cap \scr{D} _\scr{A} (A'; \vec{h} ')$. Then,
$$ 1_\scr{A} \otimes b^* A (\vec{h} (a)) ^{-1} 1 _\scr{A} \otimes c = 1_\scr{A} \otimes b^{'*} A' (\vec{h}' (a)) ^{-1} 1 _\scr{A} \otimes c'. $$ That is, $f$ is well-defined on $\mr{Dom} _\scr{A} (f) \subseteq \scr{A} ^{1\times d}$. 
\end{thm}
The following argument is essentially the same as that used in the proof of \cite[Theorem 4.4]{freefield}.
\begin{proof}
Suppose that $a = (a_1, \cdots, a_d)$ belongs to the intersection of the $\scr{A}-$domains,  $\scr{D} _\scr{A} (A ; \vec{h}) \cap \scr{D} _\scr{A} (A'; \vec{h} ')$, of the formal representations, $(A,b,c;\vec{h} ) \sim f \sim (A',b',c'; \vec{h} ')$ of $f \in \scr{M} _d$, and yet that the evaluations of these representations at $a$ do not agree,
$$ F(a) := 1_\scr{A} \otimes b^* A (\vec{h} (a)) ^{-1} 1 _\scr{A} \otimes c \neq 1_\scr{A} \otimes b^{'*} A' (\vec{h}' (a)) ^{-1} 1 _\scr{A} \otimes c' =: G(a).$$ 

Here, suppose that $A \in \C ^{(n\times n)\cdot (k+1)}$ and $A' \in \C ^{(n' \times n') \cdot (k'+1)}$ and set $m:=n+n'$. Then, consider the matrix
$$ M (\fz) := \bpm 0 & b^* & b^{'*} \\ c & A(\vec{h} (\fz)) & 0 \\ -c' & 0 & A' (\vec{h}' (\fz)) \epm \in \scr{O} _d ^{(m+1) \times (m+1)}.$$ By \cite[Proposition 5.4.6]{Cohn}, if we assume that $F(a) \neq G(a)$, it follows that the inner rank of $M(a) = \varphi _a \circ M (\fz)$ is $m +1$ over $\scr{A}$, as $\scr{A}$ is stably finite. 

On the other hand, $M (\fz ) \in \scr{O} _d  ^{(m+1)\times (m+1)} \subseteq \scr{M} _d ^{(m+1) \times (m+1)}$. Since $\scr{M} _d$ is a skew field, it is stably finite, and since $f(\fz) = F(\fz) = G(\fz ) \in \scr{M} _d$, we obtain that the inner rank of $M$ over $\scr{M} _d$ is $m$, again by \cite[Proposition 5.4.6]{Cohn}. Since $\scr{O} _d$ is a semifir and $\scr{M} _d$ is its universal skew field of fractions, the embedding $\scr{O} _d \hookrightarrow \scr{M} _d$ is honest and hence inner-rank preserving by \cite[Proposition 5.4.10]{Cohn}. We conclude that $M$ has inner rank $m$ over $\scr{O} _d$.

If $M = B C$ is an inner-rank factorization of $M$ over $\scr{O} _d$, then $B \in \scr{O} _d ^{(m+1) \times m}$ and $C \in \scr{O} _d ^{m \times (m +1)}$. But evaluation at $a \in \scr{A} ^{1 \times d}$ then yields that $M(a) = B(a) C(a) \in \scr{A} ^{(m+1) \times (m+1)}$ has a rank factorization of inner rank
$m < m+1$, a contradiction.
\end{proof}

The above result raises the natural question as to when the homomorphism $\varphi _a$, $a= (a_1, \cdots, a_d) \in \scr{A} ^{1\times d}$ is well-defined. In what follows, we construct a large class of stably-finite algebras, $\scr{A}$, including all finite $C^*-$algebras, which have the property that given any $d-$tuple of elements $a \in \scr{A} ^{1\times d}$, the unital homomorphism $\varphi _a : \scr{O} _d \rightarrow \scr{A}$ defined by $\varphi _a (\fz _i) = a_i$, $1 \leq i \leq d$, is well-defined.  

Any uniformly entire NC function, $h \in \scr{O} _d$, can be identified with a free formal power series with infinite radius of convergence, 
$$ h(\zeta) = \sum _{\om \in \word} \hat{h} _\om \fz ^\om; \quad \quad 0 = \frac{1}{R_h} = \lim _{n\rightarrow \infty} \sqrt[2n]{\sum _{|\om| = n} | \hat{h} _\om | ^2}. $$ In order to define evaluation of an entire $h \in \scr{O} _d$ in a unital, stably-finite algebra, $\scr{A}$ over $\C$, we can equip $\scr{A}$ with a suitable topology so that this infinite power series converges to an element of $\scr{A}$. For example, if $\scr{A}$ is a unital, separable and stably-finite $C^*-$algebra, then by the Gelfand--Naimark--Segal construction, we can assume, without loss of generality that $\scr{A} \subseteq \scr{B} (\cH)$ for a separable and complex Hilbert space, $\cH$. In this case given $a_1, \cdots, a_d \in \scr{A} \subseteq \scr{B} (\cH)$, the partial sums of the above FPS of $h$, evaluated at $(a_1, \cdots, a_d)$ converge absolutely and uniformly on any row-ball in $\scr{A} ^{1\times d} \subseteq \scr{B} (\cH) ^{1\times d}$ by \cite[Theorem 1.1]{Pop-freeholo}. Namely, for any fixed $r>0$, setting
$$ r \cdot \mathbb{B} ^d (\cH) := \left\{ A \in \scr{B} (\cH) ^{1\times d} | \ \| (A_1, \cdots, A_d ) \| _{\scr{B} (\cH \otimes \C ^d, \cH)} < r \right\}, $$ 
the free polynomial partial sums,
$$ h_n (\zeta) := \sum _{\ell =0} ^n  \sum _{|\om | = \ell} \hat{h} _\om \fz ^\om, $$  converge uniformly to $h(A)$ when evaluated in $r \cdot \B ^d (\cH)$, and they also converge absolutely and uniformly in $r \cdot \B^d (\cH)$ the sense that 
$$ \sum _{\ell =n} ^\infty \left\| \sum _{|\om| = \ell} \hat{h} _\om A^\om \right\| _{\scr{B} (\cH)}  \rightarrow 0,$$ as $n\rightarrow \infty$, uniformly in $r\cdot \B ^d (\cH)$. Since $\scr{A}$ is closed in the operator-norm topology and any $a \in \scr{A} ^{1\times d}$ belongs to the row-ball of finite radius, $\| a \| _{\mr{row}} \cdot \B ^d (\cH)$, it follows that $h(a) \in \scr{A}$ for any $d-$tuple $a \in \scr{A} ^{1\times d}$. 

The proof of \cite[Theorem 1.1]{Pop-freeholo} explicitly uses the $C^*-$identity for linear operators. Namely, given $A \in \scr{B} (\cH, \cJ)$, where $\cH, \cJ$ are complex Hilbert spaces, $\| A ^* A \| = \| A \| ^2$. Explicitly, if $h (\fz ) = \sum _{\om \in \word} \hat{h} _\om \fz ^\om \in \fps$ has radius of convergence $R_h >0$, and $A \in \scr{B} (\cH) ^{1\times d}$, then one estimates:
\ba
    \| h (A ) \| _{\scr{B} (\cH)} & \leq & \sum _{\ell =0} ^\infty \left\| \sum _{|\om | = \ell } \hat{h} _\om A ^\om \right\| \\
    & = & \sum _{\ell =0} ^\infty \left\| \mr{row} (A^\om )_{|\om| =\ell} \cdot  \mr{col} (\hat{h} _\om)_{|\om | = \ell} \right\| \\
    & = & \sum _{\ell =0} ^\infty \sqrt{\left\| \mr{Ad} _{A, A^*} ^{\circ \ell} (I_\cH) \right\|} \sqrt{\sum _{|\om| =\ell} | \hat{h} _\om | ^2}.
\ea
In the above, $\mr{Ad} _{A, A^*} : \scr{B} (\cH) \rightarrow \scr{B} (\cH)$ is the completely positive linear map of \emph{adjunction} by $A$ and $A^*$, 
$$ \mr{Ad} _{A,A^*} (T) := \sum _{j=1} ^d A_j T A_j ^*. $$ 
Given a $d-$tuple of bounded linear operators, $A \in \scr{B} (\cH) ^{1\times d}$ one defines the \emph{joint spectral radius} of $A$, $\rho _\sigma (A)$ as the square root of the spectral radius of the bounded linear map, $\mr{Ad} _{A,A^*}$, viewed as an element of the unital Banach algebra $\scr{B} (X)$, where $X = \scr{B} (\cH)$ \cite{Popescu-similarity}. Namely,
$$ \rho _\sigma (A) ^2 = \rho _\sigma (\mr{Ad} _{A,A^*}) = \lim _{\ell \rightarrow \infty}
\sqrt[\uproot{3} \ell ]{\left\| \mr{Ad} _{A, A^*} ^{\circ \ell} (I_\cH) \right\|}. $$
Hence, since 
$$ \frac{1}{R_h} = \limsup _{\ell \rightarrow \infty} \sqrt[\uproot{3} 2\ell]{\sum _{|\om| =\ell} |\hat{h} _\om|^2 }, $$ if the joint spectral radius of $A \in \scr{B} (\cH) ^{1\times d}$ is less than $R_h$, we have absolute convergence of the power series $h(A)$ (in operator norm).

The assumption that $\scr{A}$ is a unital $C^*-$algebra is, however, a more restrictive assumption than necessary to make sense of evaluations of $h \in \scr{O} _d$ in $\scr{A}$. To this end, we introduce a larger class of algebras, \emph{local operator algebras}, based on the construction of \emph{local operator spaces} in \cite{EffWeb,Webster}. In particular, any inverse limit of operator algebras, such as $\scr{O} _d (r)$, $r>0$, will be a local operator algebra.

\begin{defn}
Let $\scr{A}$ be a unital algebra over $\C$. A \emph{matrix semi-norm} over $\scr{A}$ is a collection of semi-norms, $\rho = (\rho _n : \scr{A} ^{n\times n} \rightarrow [0, +\infty) ) _{n=1} ^\infty$ satisfying the axioms:
\bn
\item Given any $a \in \scr{A} ^{m\times m}$ and $b\in \scr{A} ^{n\times n}$,
$$ \rho _{n+m} (a\oplus b) = \max \left\{ \rho _m (a), \rho _n (b) \right\}. $$
\item For any $\alpha \in \C ^{m \times n}$, $\beta \in \C ^{n\times m}$ and $a \in \scr{A} ^{n\times n}$, 
$$ \rho _m (\alpha a \beta) \leq \| \alpha \| \rho _n (a) \| \beta \|. $$ 
\en
We will say that multiplication in $\scr{A}$ is $\rho-$\emph{completely contractive}, or that $\rho$ is an algebra matrix semi-norm if it further obeys a third axiom: 
\bn
\item[(3)] For any $A,B \in \scr{A} ^{n\times n}$, 
$$ \rho _n (A B) \leq \rho _n (A) \cdot \rho _n (B). $$
\en
\end{defn} 

Given an algebra matrix semi-norm, $\rho$, on $\scr{A}$, let
$$ N _\rho := \{ a \in \scr{A} | \ \rho _1 (a) =0 \}. $$ 
This is clearly a two-sided ideal in $\scr{A}$, so that the unit, $1$, of $\scr{A}$ belongs to $N_\rho$ if and only if $\rho$ is identically $0$. Hence, assuming $\rho \neq 0$, $A_\rho := A/ N_{\rho}$ is again a unital and stably finite algebra with well-defined multiplication,
$$ (a+N_\rho)(b+N_\rho) = ab + N_\rho. $$ Moreover, if we define 
$$ \| a + N_\rho \|_\rho := \rho (a), $$ then $\| \cdot \|_\rho$ is a matrix-norm on $\scr{A} _\rho$ in the sense of Ruan, and moreover if multiplication in $\scr{A}$ is completely $\rho-$contractive then for any $A,B \in \scr{A} _\rho ^{n\times n}$, 
$$ \| A \cdot B \| _\rho \leq \| A \| _{\rho} \| B \| _{\rho} = \rho _n (A) \rho _n (B), $$ so that ``multiplication is completely contractive" and  $(A_\rho, \| \cdot \| _\rho)$ is an abstract operator algebra in sense of Blecher, Ruan and Sinclair \cite[Corollary 16.7]{Paulsen-cbmaps}. We will further assume that our matrix semi-norms are `unital' in the sense that $\rho (1) = 1$. In this case $(A_\rho, \| \cdot \| _\rho)$ is a unital abstract operator algebra, and hence is completely isometrically isomorphic to a unital algebra of operators on a Hilbert space by the Blecher--Ruan--Sinclair theorem.

\begin{defn}
A (unital) \emph{local operator algebra} is a unital and separable topological algebra, $\scr{A}$, over $\C$, equipped with a family of unital and matrix-multiplicative semi-norms that generate the topology on $\scr{A}$.  
\end{defn}

The inverse limit of any family of unital operator algebras will be a local operator algebra. For example, each of the ball semifirs of uniformly analytic NC functions constructed in this paper, $\scr{O} _d (r)$, $r>0$, are local operator algebras. Conversely, any local operator algebra can be viewed as an inverse limit of abstract operator algebras \cite{Webster}. If $\scr{A}$ is a local operator algebra, $a=(a_1, \cdots, a_d) \in \scr{A} ^{1\times d}$ and $h \in \scr{O} _d$, the free polynomial partial sums, $h_n (a)$ then converge to an element, $h(a)$, in $\scr{A}$. Indeed, to prove this, one need only show that given any $\rho \in \La$, that the sequence $h_n (a + N_\rho)$ is Cauchy in the abstract, unital and separable operator algebra $A/N_\rho$, and this follows immediately from \cite[Theorem 1.1]{Pop-freeholo}, since $h$ has infinite Cauchy--Hadamard radius of convergence and we can view $A/N_\rho$ as an algebra of operators on a separable, complex Hilbert space. That is, if $\scr{A}$ is a unital local operator algebra, then for any $d-$tuple $a=(a_1, \cdots, a_d) \in \scr{A} ^{1\times d}$, the unital homomorphism, $\varphi _a : \scr{O} _d \rightarrow \scr{A}$ is well-defined. Hence, if a unital local operator algebra is also stably-finite, Theorem~\ref{thm:stabeval} applies and shows that we have well-defined $\scr{A}-$domains and evaluations for any $f \in \scr{M} _d$. That is, we have shown the following.

\begin{cor}
Let $\scr{A}$ be a unital local operator algebra. Then for any $a = (a_1, \cdots, a_d) \in \scr{A} ^{1\times d}$, there is a well-defined unital homomorphism $\varphi _a : \scr{O} _d \rightarrow \scr{A}$ obeying $\varphi _a (\fz _i) = a_i$. If $\scr{A}$ is also stably-finite, then any $f \in \scr{M} _d$ has a well-defined $\scr{A}-$domain and evaluations,
$$ \mr{Dom} _{\scr{A}} (f) := \{ a \in \scr{A} ^{1\times d} | \ f \sim (A,b,c;\vec{h}) \ \mbox{and} \ A(\vec{h} (a))^{-1} \, \exists \}, $$
$$ f(a) = 1_\scr{A} \otimes b^* A (\vec{h} (a)) ^{-1} 1_\scr{A} \otimes c; \quad \quad a \in \mr{Dom} _\scr{A} (f).$$
\end{cor}

In particular, each $\scr{O} _d (r)$, $r \in (0, +\infty]$ is a stably-finite local operator algebra. 

\subsection{Strong convergence of meromorphic functions}

Let $(\cA, \tau)$ be a tracial $C^*-$algebra, \emph{i.e.} a unital $C^*-$algebra equipped with a faithful and normalized (unital) trace, $\tau \in \cA ^*$. Recall from \cite{Hup-realize2} that a tuple of elements $s_1,\cdots, s_d \in \cA$, is called \emph{free}, if for every $p_1, \cdots, p_n \in \fp$, $n\in \N$, and every word $\om = i_1 \cdots i_n \in \word$, obeying $i_k \neq i_{k+1}$, the assumption that $\tau(p_j(s_{i_j})) = 0$ implies that $\tau(p_1(s_{i_1}) p_2(s_{i_2}) \cdots p_n(s_{i_n})) = 0$. Such a free family, $s=(s_1, \cdots, s_d)$, is called a \emph{free semicircular family}, if for every $k \in \N$ and every $1 \leq j \leq d$, we have both $s_j = s_j^*$ and 
\[
    \tau(s_j^k) = \begin{cases} \frac{2}{k+2} \binom{k}{k/2}, & k \text{ is even}; \\ 0, & k \text{ is odd}. \end{cases}
\]
Let us fix a free semicircular tuple $s_1,\cdots,s_d$. Let $\sgrm(n,\tfrac{1}{n})$ denote the class of all $n \times n$ self-adjoint random matrices on a probability space $\Omega$ with independent and identically distributed (i.i.d.) Gaussian entries with mean $0$ and variance $\frac{1}{n}$. This normalization of the variance goes back to Voiculescu \cite{Voic-limit_laws}. In \cite[Theorem A]{Hup-realize2} Haagerup and Thorbj{\o}rnsen proved that if we take a sequence $(X_1^{(n)},\cdots,X_d^{(n)})$ of $d$-tuples of independent elements of $\sgrm(n,\tfrac{1}{n})$ for all $n \in \N$, then there exists a measure zero set $\Xi \subset \Omega$, such that for every $p \in \C \langle z \rangle$ and every $\omega \in \Omega \setminus \Xi$,
\[
\lim_{n\to \infty} \left\| p(X_1^{(n)}(\omega),\cdots,X_d^{(n)}(\omega) ) \right\| = \|p(s_1,\cdots,s_d)\|.
\]
This result was extended to matrix-valued polynomials by Haagerup, Schultz, and Thorbj{\o}rnsen in \cite{Hup-realize}. Using the terminology of \cite{Yin-strong_conv}, we will say that the sequence $(X_1^{(n)}, \cdots, X_d^{(n)})$ \emph{converges strongly} to the free semicircular tuple $s_1,\cdots,s_n$. In \cite[Corollary 1]{Yin-strong_conv}, Yin showed that if $\fr$ is a rational expression with the tuple $(s_1,\cdots,s_d)$ in its domain, then for $n$ big enough and almost every $\omega \in \Omega$, $(X_1^{(n)}(\omega),\cdots,X_d^{(n)}(\omega))$ lies in the domain of $\fr$ and 
\[
\lim_{n \to \infty} \left\| \fr (X_1^{(n)}(\omega),\cdots,X_d^{(n)}(\omega)) \right\| = \|\fr (s_1,\cdots,s_d)\|.
\]
It follows that,
\[
\lim_{n\to \infty} \mathbb{E} \left(\frac{1}{n} \tr(\fr (X_1^{(n)},\cdots,X_d^{(n)}) ) \right) = \tau(\fr (s_1,\cdots,s_n)).
\]
Here, $\mathbb{E}$ stands for the expectation of a random variable. This result extends the result of Haagerup and Thorbj{\o}rnsen to the free skew field. Our goal is to extend these results even further to the universal skew field of fractions, $\scr{M} _d$, of $\scr{O} _d$. Note that by the result of Haagerup and Thorbj{\o}rnsen, the row tuple, $(X_1^{(n)}(\omega), \cdots, X_d^{(n)}(\omega))$ is, almost surely, in a row ball of bounded radius. We start with a simple lemma.

\begin{lem} \label{lem:strong_conv_entire}
   Assume that $k=1$. Let $f \in \scr{O} _d ^{k \times k}$, $k \in \N$. Then, for almost every $\omega \in \Omega$,
\[
\lim_{n \to \infty} \left\| f(X_1^{(n)}(\omega),\cdots,X_d^{(n)}(\omega)) \right\| = \|f(s_1,\cdots,s_d)\|.
\]
\end{lem}
\begin{proof}
Any such $f$ is given by a free FPS, 
\[
f(\fz ) = \sum_{k=0}^{\infty} \sum_{|\alpha|=k} \hat{f}_{\alpha} \, \fz ^{\alpha},
\]
with infinite radius of convergence.
Let $p_m = \sum_{k=0}^m \sum_{|\alpha|=k} \hat{f} _{\alpha} \fz ^{\alpha}$ be the $m$-th free polynomial partial sum of $f$. Then, this sequence of free polynomial partial sums, $p_m$, converges uniformly to $f$ on any NC row-ball of finite radii. By \cite[Theorem A]{Hup-realize2}, for almost all $\omega \in \Omega$ and all $m \in \N$, 
\[
\lim_{n\to \infty} \left\| p_m(X_1^{(n)}(\omega),\cdots,X_d^{(n)}(\omega)) \right\| = \|p_m(s_1,\cdots,s_m)\|.
\]
We see that 
\ba   
& & \left| \| f(X_1^{(n)}(\omega),\cdots,X_d^{(n)}(\omega)) \|  - \|f(s_1,\cdots,s_d)\| \right|  \\ & \leq &  \left\| f(X_1^{(n)}(\omega),\cdots,X_d^{(n)}(\omega)) - p_m(X_1^{(n)}(\omega),\cdots,X_d^{(n)}(\omega)) \right\| \\ & +  &  \left| \| p_m(X_1^{(n)}(\omega),\cdots,X_d^{(n)}(\omega)) \| - \|p_m(s_1,\cdots,s_d)\| \right| + \|p_m(s_1,\cdots,s_d) - f(s_1,\cdots,s_d)\| \stackrel{m \rightarrow \infty}{\longrightarrow} 0.
\ea
One argues similarly in the matrix-valued case.
\end{proof}

The following lemma is a slight modification of the corresponding result in \cite{Yin-strong_conv} to the setting at hand.
\begin{lem} \label{lem:invertible}
For every $k \in \N$ and every full matrix $F \in \scr{O}_d ^{k\times k}$, if $F(s_1,\cdots,s_d)$ is invertible, then for almost every $\omega \in \Omega$ for $n$ sufficiently large, $F(X_1^{(n)}(\omega),\cdots,X_d^{(n)}(\omega))$ is invertible.
\end{lem}
\begin{proof}
There is a natural anti-holomorphic involution on the free algebra $\fp$, with respect to which the formal NC variables are self-adjoint (namely, $\fz_j^* = \fz_j$ for all $1 \leq j \leq d$).  The involution is given by transposing the monomials and taking the complex conjugate of the coefficients, 
$$  \sum _{\om \in \word} \hat{p} _\om \, \fz ^\om \quad \mapsto \quad \sum \ov{\hat{p} _{\om}} \, \fz ^{\om ^\mrt}. $$
This involution extends naturally to $\scr{O}_d$ and to matrices over $\scr{O}_d$. Let $F \in \scr{O}_d ^{k\times k}$, then $F^* F \in \scr{O}_d ^{k\times k}$ is self-adjoint. Moreover, for every tuple of self-adjoint elements $Y_1,\cdots,Y_d \in B(\cH)$, we have that $(F^* F)(Y_1,\cdots , Y_d) = F(Y_1,\cdots,Y_d)^* F(Y_1,\cdots,Y_d)$. To simplify notations, let $T = F(s_1,\cdots,s_d)$ and $T_n(\omega) = F(X_1^{(n)}(\omega),\cdots,X_d^{(n)}(\omega))$. We note that an operator $A \in W^*(s_1,\ldots,s_d)$ is invertible if and only if $A^* A$ is invertible if and only if $\big\lVert \lVert A^2\rVert I - A^* A \big\rVert < \|A\|^2$. By the previous lemma $\|T_n(\omega)\| \to \|T\|$ almost surely. Moreover, almost surely
\[
    \lim_{n \to \infty} \normBig{\norm{T}^2 I - T_n(\omega)^* T_n(\omega)}
    = \normBig{\norm{T}^2 I - T^* T},
\]
and
\ba
    \normBig{\norm{T}^2 I - T_n(\omega)^* T_n(\omega)}
    & \leq & \Big\lvert \norm{T_n(\omega)}^2 - \norm{T}^2 \Big\rvert  + \normBig{ \lVert T\rVert^2 I - T_n(\omega)^* T_n(\omega)} \\ 
    & \stackrel{n \rightarrow \infty}{\longrightarrow} & \normBig{\norm{T}^2 I - T^* T} < \norm{T}^2. 
\ea
Therefore, for almost every $\omega$ and $n$ big enough, $\Norm{\norm{T_n(\omega)}^2 I - T_n(\omega)^* T_n(\omega)} < \norm{T_n}^2$ and $T_n$ is invertible.
\end{proof}

The following lemma is an immediate consequence of the fact that the group of invertible operators on a complex Hilbert space is operator--norm open and the inversion map is uniformly continuous on this group.
\begin{lem} \label{lem:rat_approx}
    Let $F \in \scr{O}_d ^{k\times k}$ be a full matrix. Let $Q_m \in \fp ^{k\times k}$ be the partial sums of the free FPS expansion of $F$ around the origin. Then for every $X = (X_1,\cdots,X_d) \in \scr{B} (\cH)^d$ so that $F(X)$ is invertible, $Q_m(X)$ is invertible for $m$ big enough. Moreover, for every such $X$, there exists $r > 0$, such that $Q_m^{-1} \to F^{-1}$ uniformly on the closed uniform NC row-ball of radius $r$ centered at $X$.
\end{lem}

It follows immediately that if $f$ is an NC rational expression in $\scr{M}_d$, then there exists a sequence of rational expressions $\fr _m$, such that $\fr _m \to f$ uniformly on uniform balls centered at points in the domain of the rational expression. Indeed, we can approximate $b$, $c$, and $d$ uniformly on subballs by polynomial vectors and the previous lemma takes care of the rest.

\begin{thm} \label{thm:meromorphic_hagthor}
Let $f \in \scr{M}_d$ and assume that $(s_1,\cdots,s_d)$ are in the domain of $f$. Then, for almost every $\omega \in \Omega$,
\[
\lim_{n \to \infty} \left\| f(X_1^{(n)}(\omega),\cdots,X_d^{(n)}(\omega)) \right\| = \|f(s_1,\cdots,s_d)\|.
\]
\end{thm}
\begin{proof}
By Theorem \ref{thm:stabeval} and its preceding discussion, 
we can choose a representation of $f = b^* A^{-1} c$, such that $b, c \in \C^k$, $A \in \scr{O}_d ^{k \times k}$ is full, and $A(s_1,\cdots,s_d)$ is invertible. By Lemma \ref{lem:invertible}, $A(X_1^{(n)}(\omega),\cdots,X_d^{(n)}(\omega))$ is invertible, for almost every $\omega$ and $n$ big enough. Now fix $\omega \in \Omega \setminus \Xi$, where $\Xi$ is the measure $0$ set that we must remove to ensure convergence. By the observation preceding the statement of the theorem, there exists a sequence of rationals $\fr _m$ that converges to $f$ uniformly on uniform balls centered at points in the domain of $f$. It remains to perform a calculation. For $m$ big enough, the following inequality makes sense

\begin{align*} 
&\Abs{\norm{f(s_1,\cdots,s_d)}  -  \norm{f(X_1^{(n)}(\omega),\cdots,X_d^{(n)}(\omega))}}
     \leq \norm{f(s_1,\cdots,s_d) - \fr_m (s_1,\cdots,s_d)} \\ 
    &+ \Abs{\norm{\fr _m (s_1,\cdots,s_d) } - \norm{\fr_m (X_1^{(n)}(\omega),\cdots,X_d^{(n)}(\omega))} } \\
    &+ \norm{\fr _m(X_1^{(n)}(\omega),\cdots,X_d^{(n)}(\omega)) - f(X_1^{(n)}(\omega),\cdots,X_d^{(n)}(\omega))}. 
\end{align*}
Given $\varepsilon > 0$, we can choose $m$ big enough so that
\ba
    & & \Abs{ \|f(s_1,\cdots,s_d)\| - \|f(X_1^{(n)}(\omega),\cdots,X_d^{(n)}(\omega))\| } \\ 
    & \leq & \varepsilon + \Abs{ \norm{\fr _m(s_1,\cdots,s_d)} - \norm{\fr _m(X_1^{(n)}(\omega),\cdots,X_d^{(n)}(\omega))} }.
\ea
By the result of Yin \cite{Yin-strong_conv}, letting $n$ tend to infinity, we obtain 
\[
\limsup_{n\to \infty} \Abs{\norm{f(s_1,\cdots,s_d)} - \norm{f(X_1^{(n)}(\omega),\cdots,X_d^{(n)}(\omega))}} \leq \varepsilon.
\]
Since this holds for every $\varepsilon > 0$, the theorem is proven.
\end{proof}

As in \cite{Yin-strong_conv}, this has the immediate consequence:
\begin{cor}
Given any $f \in \scr{M}_d$ so that $(s_1,\cdots,s_d)$ are in the domain of $f$, 
\[
\lim_{n\to \infty} \mathbb{E} \left( \frac{1}{n} \tr \, f (X_1^{(n)},\cdots,X_d^{(n)}) \right) = \tau \left( f(s_1,\cdots,s_n) \right).
\]
\end{cor}

\section{Outlook} \label{outlook}

We leave several interesting questions open. One natural question arises from the fact that the Fr\'echet topology on $\scr{O}_d$, viewed as an inverse limit of NC disk algebras, is independent of the chosen operator space structure on $\C^d$. Here, we have chosen to work with the row operator space structure over $\C ^d$ to exploit non-commutative Hardy space techniques, since the free Hardy space can be viewed as a Hilbert space of uniformly analytic NC functions on the NC unit row-ball centred at $0 \in \C ^{1\times d}$.  However, our methods depend crucially on the `row operator space structure' that is implicit in the definition of Popescu's NC disk algebra as NC functions in the NC unit row-ball. (The column operator space structure for $\C ^d$ would work equally well.) If $\cE$ is an operator space structure on $\C^d$, then we denote by $\B_{\cE} ^d$ the NC set that is the open unit ball of $\cE$. The algebras $\bH^{\infty}(\B_{\cE} ^d )$ and $\A(\cE)$ were studied by Sampat and Shalit \cite{SampShal-classification,SampShal-structure}. One can ask whether the algebras $\scr{O}(R \cdot \B_{\cE} ^d)$ given by NC functions that are uniformly bounded on $r \cdot \overline{\B_{\cE}^d}$ for every $0 < r < R$ are semifirs that satisfy an analogue of the analytic Bergman Nullstellensatz. A second natural question is whether every closed right (respectively, left) ideal in $\scr{O}_d$ is free as a right (respectively, left) $\scr{O}_d$-module in the sense of completed tensor products as in \cite{Taylor2}.

In \cite{AugMarS}, the authors have introduced a second skew field of NC meromorphic functions, $\scr{M} _{d} ^\scr{C}$. This is the universal skew field of fractions of the semifir, $\scr{O} _{d} ^\scr{C}$, of all NC functions that admit a compact realization, $(A,b,c) \in \scr{C} (\cH) ^d \times \cH \times \cH$, where $\cH$ is a complex, separable Hilbert space and $\scr{C} (\cH)$ denotes the compact linear operators on $\cH$. That is, given any $h \in \scr{O} _d ^\scr{C}$, given by the compact realization $(A,b,c)$, and any $Z \in \cdn$ for which the monic, affine--linear pencil, $L_A (Z) := I_n \otimes I_\cH - \sum _{j=1} ^d Z_j \otimes A_j$, is invertible, 
$$ h(Z) = I_n \otimes b^* L_A (Z) ^{-1} I_n \otimes c. $$ By \cite[Theorem 4.5 and Corollary 4.4]{AugMarS}, any free formal power series, $f \in \fpsd$, is a uniformly entire NC function in $\scr{O} _d$ if and only if it has a compact and jointly quasinilpotent realization, $(A,b,c)$. That is, each component operator, $A_j \in \scr{C} (\cH)$ is compact and the completely positive map, $\mr{Ad} _{A,A^*} : \scr{B} (\cH) \rightarrow \scr{B} (\cH)$, of adjunction by the $A_j$ and $A_j ^*$, is quasinilpotent as an element of the unital Banach algebra, $\scr{B} (X)$, where $X:= \scr{B} (\cH)$. In particular, $\scr{O} _d \subsetneqq \scr{O} _d ^\scr{C}$. It then follows from the standard realization algorithm for sums, products and inverses of NC functions with realizations that a univariate formal power series is the Taylor series of a globally meromorphic function (which is analytic at $0$) if and only if it has a compact realization \cite[Theorem 5.1 and Theorem 5.3]{AugMarS}. This motivates the interpretation of $\scr{O} _d ^\scr{C}$, the semifir of all NC functions with compact realizations, as the ring of globally meromorphic NC functions with analytic germs at $0$, as well as the interpretation of its universal skew field, $\scr{M} _d ^\scr{C}$ as the skew field of globally meromorphic NC functions. Here, elements of $\scr{M} _d ^\scr{C}$ can indeed be identified with uniformly analytic NC functions whose domains are uniformly open subsets that are level-wise analytic-Zariski dense at all sufficiently high levels in the NC universe, see \cite[Section 6.1 and Theorem 6.4]{AugMarS}.

In one-variable, it is not difficult to see, from the results of \cite{AugMarS}, that these two definitions of the field of globally meromorphic functions are the same, \emph{i.e.} that $\scr{M} _1 = \scr{M} _1 ^\scr{C}$. 
Therefore, it would be very interesting to know whether these two skew fields of NC meromorphic functions, $\scr{M} _d$ and $\scr{M} _d ^\scr{C}$ are actually the same field for any $d \in \N$. 
Since, as discussed above, any $h \in \scr{O} _d$ has a compact and jointly quasinilpotent realization, we have an embedding $\scr{O} _d \hookrightarrow \scr{O} _d  ^\scr{C}$.
By Theorem~\ref{thm:NC Entire Amitsur}, $\scr{O} _d  \hookrightarrow \fps$ is honest, hence $\scr{M} _d \hookrightarrow \scr{M} _d ^\scr{C}$, so that $\scr{M} _d$ is a skew sub-field of $\scr{M} _d ^\scr{C}$. 
A special case of a classical question posed by Henri Poincar\'e, and ultimately answered, in the affirmative by the positive solution to the Cousin problems, is whether any globally meromorphic function in $\C ^d$ is a ratio of entire functions \cite{Maurin}. 
Thus, we arrive at an ``NC Poincar{\'e} problem":
\begin{quest}
    For any $d\in \N$, is $\scr{M} _d = \scr{M} _d ^\scr{C}$?
    That is, is any uniformly meromorphic NC function in $\scr{M}_d ^\scr{C}$ equal to an NC rational expression in NC entire functions?
\end{quest}

\section*{Conflict of interest}

On behalf of all authors, the corresponding author states that there is no conflict of interest.

\setstretch{1}
\setlength{\parskip}{0pt}
\setlength{\itemsep}{0pt}

\bibliographystyle{abbrv}
\bibliography{entire}

\end{document}